\documentclass[12pt]{amsart}
\usepackage{anysize}
\marginsize{2.7cm}{2.44cm}{2.0cm}{3.0cm}
\usepackage[small,it]{caption}
\usepackage[OT2,T1]{fontenc}
\DeclareSymbolFont{cyrletters}{OT2}{wncyr}{m}{n}
\DeclareMathSymbol{\Sha}{\mathalpha}{cyrletters}{"58}
 \usepackage[latin1]{inputenc}
 \usepackage[dvips]{graphicx}
 \usepackage{wrapfig}
 \usepackage{amsmath}
 \usepackage{amsthm}
 \usepackage{amsfonts}
 \usepackage{amssymb}
 \usepackage{layout
} \usepackage{verbatim}
 \usepackage{alltt}
\usepackage{yfonts}
 \usepackage{pifont}

\usepackage[all]{xy}
\usepackage{setspace}

\newtheorem*{thma}{Theorem A}
\newtheorem*{thmb}{Theorem B}

\newtheorem*{corc}{Corollary C}

\newcommand{\LL}{\Lambda}

\newcommand{\QQ}{\mathbb{Q}}
\newcommand{\FF}{\mathcal{F}}

\newcommand{\lra}{\longrightarrow}
\newcommand{\ZZ}{\mathbb{Z}}

\newcommand{\ra}{\rightarrow}

\newcommand{\be}{\begin{equation}}
\newcommand{\ee}{\end{equation}}

\newcommand{\al}{\mathcal{L}}

\newcommand{\FFc}{\mathcal{F}_{\textup{\lowercase{can}}}}

\numberwithin{equation}{section}
\newtheorem{thm}{Theorem}[section]
\newtheorem{lemma}[thm]{Lemma}
\newenvironment{define}{\par\medskip\noindent\refstepcounter{thm}
\bgroup{\hspace*{-0.15 cm}\bf{Definition}
\thethm.}\bgroup}{\egroup \egroup\par\medskip}\newtheorem{prop}[thm]{Proposition}
\newtheorem{cor}[thm]{Corollary}
\newenvironment{rem}{\par\medskip\noindent\refstepcounter{thm}
\bgroup{\hspace*{-0.15 cm}\bf{Remark} \thethm.}\bgroup}{\egroup
\egroup\par\medskip} \parskip 2pt

\newenvironment{conj}{\par\medskip\noindent\refstepcounter{thm}
\bgroup{\hspace*{-0.15 cm}\bf{Conjecture}
\thethm.}\bgroup}{\egroup \egroup\par\medskip}

\newcounter{Athm}[section]

\renewcommand{\theAthm} {\arabic{Athm}}

\long\def\symbolfootnote[#1]#2{\begingroup%
\def\thefootnote{\fnsymbol{footnote}}\footnote[#1]{#2}\endgroup}

\begin{document}
\title{O\lowercase{n} N\lowercase{ekov\'a\v{r}'s heights, exceptional zeros and a conjecture of} M\lowercase{azur}-T\lowercase{ate}-T\lowercase{eitelbaum}}

\author{K\^az\i m B\"uy\"ukboduk}

\email{kazim@math.stanford.edu}
\address{K\^az\i m B\"uy\"ukboduk \hfill\break\indent {Ko\c{c} University Mathematics, 34450 Sariyer, Istanbul, Turkey}}

\begin{abstract}
Let $E/\QQ$ be an elliptic curve which has split multiplicative reduction at a prime $p$ and whose analytic rank $r_{\textup{an}}(E)$ equals one. The main goal of this article is to relate the second order derivative of the Mazur-Tate-Teitelbaum $p$-adic $L$-function $L_p(E,s)$ of $E$ to Nekov\'{a}\v{r}'s height pairing evaluated on natural elements arising from the Beilinson-Kato elements. Along the way, we extend \emph{a} Rubin-style formula of Nekov\'a\v{r} (or in an alternative wording, correct \emph{another} Rubin-style formula of his) to apply in the presence of exceptional zeros. Our height formula allows us, among other things, to compare the order of vanishing of $L_p(E,s)$ at $s=1$ to its (complex) analytic rank $r_{\textup{an}}(E)$ assuming the non-triviality of the height pairing. This has consequences towards a conjecture of Mazur, Tate and Teitelbaum. 
\end{abstract}

\maketitle
\tableofcontents
\section{Introduction}
\label{sec:intro}
 Fix a prime $p>3$ and an elliptic curve $E$ defined over $\QQ$ that has split multiplicative reduction at $p$. Let $L(E,s)$ (resp., $L_p(E,s)$) denote the complex Hasse-Weil $L$-function (resp., the Mazur-Tate-Teitelbaum $p$-adic $L$-function) attached to $E$. By the work of Wiles $L(E,s)$ is admits an analytic continuation to the whole complex plane. Let $r_{\textup{an}}(E)$ denote the order of vanishing of $L(E,s)$ at $s=1$. As we have assumed that the elliptic curve $E$ has split multiplicative reduction at $p$, the $p$-adic $L$-function $L_p(E,s)$ has an \emph{exceptional zero} at $s=1$ in the sense of Greenberg~\cite{gr} due to the vanishing of the interpolation factor $(1-p^{1-s})(1-p^{-s})$ at $s=1$. Mazur, Tate and Teitelbaum conjecture in this case that
 \be\label{eqn:conjmtt1}
 \textup{ord}_{s=1}\, L_p(E,s)=1+r_{\textup{an}}(E).
 \ee
 This is the conjecture that the title of this article refers to. Furthermore, they conjectured a formula for the first derivative of $L_p(E,s)$:
 \be\label{eqn:mttconj2}
 \frac{d}{ds}L_p(E,s)\Big|_{s=1}=\frac{\log_p(q_E)}{\textup{ord}_p(q_E)}\cdot\frac{L(E,1)}{\Omega_E^+},
 \ee
 where $\Omega_E^+$ is the real period of $E$ and $q_E$ is the Tate period of $E$ (obtained via the $p$-adic uniformization of $E$) and $\log_p$ is the $p$-adic logarithm. Greenberg and Stevens \cite{greenbergstevens} gave a proof of the assertion (\ref{eqn:mttconj2}).  The so-called \emph{Saint-Etienne theorem} (formerly, a conjecture of Manin) proved in~\cite{sainetienne} shows that $\log_p(q_E)\neq0$. We therefore conclude that (\ref{eqn:conjmtt1}) holds true when $r_{\textup{an}}(E)=0$. As far as the author is aware, nothing substantial was known when $r_{\textup{an}}(E)> 0$ prior to this work.

The conjecture of Birch and Swinnerton-Dyer (henceforth, abbreviated as BSD) predicts that the behavior of the Hasse-Weil $L$-function $L(E,s)$ at $s=1$ is related to the ($p$-adic) Selmer group $\textup{Sel}_p(E/\QQ)$ (see \S\ref{subsubsec:comparison} below for a definition of this Selmer group). In particular, BSD predicts that $r_{\textup{an}}(E)=\textup{rank}_{\ZZ_p}(\textup{Sel}_p(E/\QQ))$ and  further that the $r_{\textup{an}}(E)$-th derivative of $L(E,s)$ at $s=1$ should be expressed (among other things) in terms of a certain regulator calculated on $\textup{Sel}_p(E/\QQ)$.

The conjectured equality (\ref{eqn:conjmtt1}) suggests that, in order to formulate the $p$-adic analog of BSD for $L_p(E,s)$ at $s=1$ one should replace the classical Selmer group with an extended Selmer group so as to compensate for the (conjectural) gap between the rank of $\textup{Sel}_p(E/\QQ)$ and $\textup{ord}_{s=1}\, L_p(E,s)$. This has been carried out initially in \cite{mtt}; later Nekov\'a\v{r} in \cite{nek} defined his \emph{extended} Selmer groups in a much more general framework. The purpose of this article is to express the first (resp., second) order derivative of $L_p(E,s)$ at $s=1$ when $r_{\textup{an}}(E)=0$ (resp., when $r_{\textup{an}}(E)=1$) in terms of Nekov\'{a}\v{r}{s}'s height pairings defined on his extended Selmer groups. When $r_{\textup{an}}(E)=0$, this allows us to interpret Kobayashi's computations \cite{kobayashi06} from the perspective offered by Nekov\'a\v{r}'s general theory. The main contribution of this article, however, concerns the case $r_{\textup{an}}(E)=1$. In this case, relying on a Rubin-style formula we prove\footnote{In an earlier version of this article we have made use of the (incorrect) Proposition~11.5.11 of \cite{nek}. However, we are still able to prove a statement along these lines (which might be of independent interest) befitting our needs in Appendix A below.} in the appendix we reduce the conjecture (\ref{eqn:conjmtt1}) to the non-degeneracy of Nekov\'a\v{r}'s $p$-adic height pairing. 

\begin{rem}
Soon after we circulated the initial version of this article among experts in late 2012 (in which we had assumed the truth of a conjecture of Perrin-Riou), we learned about R. Venerucci's work on Perrin-Riou's conjecture, which then allowed us to lift that hypothesis on our results. We remark that R. Venerucci has subsequently (yet completely independently) deduced the main results of this article albeit in a somewhat different form.
\end{rem}

Before we explain the results of the current article in detail, let us introduce some notation. See also \cite{kbbheights} for an investigation along these lines when $E$ is replaced by $\mathbb{G}_\textup{m}$ and when the relevant $p$-adic $L$-function is the Kubota-Leopoldt $p$-adic $L$-function. 
\subsection*{Acknowledgements} I thank Barry Mazur heartily for having a close look at this work and Denis Benois, Tadashi Ochiai, Karl Rubin for helpful correspondences and comments. I am also indebted to Masato Kurihara who suggested that I should compare Kobayashi's work with the results of \cite{kbbheights} which initiated the train of thoughts that led me to this work and to Massimo Bertolini who has notified me about the work of R. Venerucci on Perrin-Riou's conjecture. I thank CRM for their hospitality; Francesc Castella, Victor Rotger (and his group in Barcelona) for organizing a seminar where I explained this work in detail and for their feedbacks during the series of talks. Finally, I am deeply grateful to R. Venerucci for carefully  reading this manuscript and pointing out many inaccuracies in an earlier version of this article. 

When this project was carried out, the author was partially supported by the Marie Curie grant EC-FP7 230668, a T\"UB\.ITAK grant and by the Turkish Academy of Sciences.
\subsection{Notation and Hypotheses}
\label{subsec:notation}
 For any  field $K$, fix a separable closure $\overline{K}$ of $K$ and set  $G_K=\textup{Gal}(\overline{K}/K)$.  Let $\QQ_\infty/\QQ$ denote the cyclotomic $\ZZ_p$-extension of $\QQ$ and let $\Gamma=\textup{Gal}(\QQ_\infty/\QQ)$. We write $\rho_{\textup{cyc}}$ for the cyclotomic character $\rho_{\textup{cyc}}:\Gamma\stackrel{\sim}{\ra}1+p\ZZ_p$. Let $\QQ_n$ denote the unique sub-extension of $\QQ_\infty/\QQ$ of degree $p^n$ over $\QQ$, i.e., the fixed field of $\Gamma^{p^n}$. Let $\Phi_n$ be the completion of $\QQ_n$ at the unique prime of $\QQ_n$ above $p$, and set $\Phi_\infty=\cup \Phi_n$, the cyclotomic $\ZZ_p$-extension of $\QQ_p$. By slight abuse of notation $\textup{Gal}(\Phi_\infty/\QQ_p)$ will be denoted by $\Gamma$ as well. Let $\Gamma_n=\Gamma/\Gamma^{p^n}=\textup{Gal}(\QQ_n/\QQ)$. We fix a topological generator $\gamma$ of $\Gamma$. We also set $\LL=\ZZ_p[[\Gamma]]$ as the cyclotomic Iwasawa algebra and $J=\ker(\LL\ra\ZZ_p)$ (where the arrow is the map induced from $\gamma\mapsto1$) as the augmentation ideal.

 Let $E/{\QQ}$ be an elliptic curve that has split multiplicative reduction at $p$ and let $S\supset \{p,\infty\}$ denote the set of primes where $E$ has bad reduction. Let $T=T_p(E)$ denote its $p$-adic Tate module and set $V=T\otimes\QQ_p$. We have an exact sequence
 \be\label{eqn:tateexactseq} 0\lra F^+_pT\lra T \lra F^-_pT\lra 0\ee
 of $\ZZ_p[[G_{\QQ_p}]]$-modules, where $F^+_pT\cong \ZZ_p(1)$ and $F^-_p T \cong \ZZ_p$. Let $T^*=\textup{Hom}(T,\ZZ_p(1))$ (resp., $V^*=T^*\otimes\QQ_p$) and $F^{\pm}T^*=\textup{Hom}(F^{\mp}_pT,\ZZ_p(1))$, so that $T^*$ fits in an exact sequence of $\ZZ_p[[G_{\QQ_p}]]$-modules
$$0\lra F^+_pT^*\lra T^* \lra F^-_pT^*\lra 0.$$
Note that the Weil pairing shows that there is an isomorphism $T\cong T^*$ of $\ZZ_p[[G_\QQ]]$-modules.
Let $\tan(E/\QQ_p)$ denote the tangent space of $E/\QQ_p$ at the origin and consider the Lie group exponential
map
$$\exp_{E}: \tan(E/\QQ_p) \lra E(\QQ_p)\otimes\QQ_p.$$
Fix a minimal Weierstrass model of $E$ and let $\omega_E$ denote the corresponding holomorphic
differential. The cotangent space $\textup{cotan}(E)$ is generated by the invariant differential $\omega_E$, let $\omega_E^*\in \tan(E/\QQ_p)$ be the corresponding dual basis. Then there is a dual exponential map
$$\exp^*_E: H^1(G_p,V^*) \lra \textup{cotan}(E)=\QQ_p\omega_E$$
and an induced map
$$\exp_{\omega_E}^*=\omega_E^*\circ \exp_E^*: H^1(G_p,V^*) \lra \QQ_p.$$
Let $E_p(s)=1-p^{-s}$ denote the Euler factor of $L(E,s)$ at $p$ and define
 $$\xymatrix{\rho: \Gamma \ar[r]^(.47){{\rho_{\textup{cyc}}}}&1+p\ZZ_p\ar[rr]^(.55){E_p(1)^{-1}\log_p}&&\ZZ_p}$$
 to be a fixed normalization of $\rho_{\textup{cyc}}$.
 \subsection{Statements of the results}
 \label{subsec:results}
 For $X=V,V^*$, let $\widetilde{H}^1_f(X)$ denote Nekov\'a\v{r}'s extended Selmer group attached to $X$ and let
\be\label{eqn:nekpairingintro}\langle\,,\,\rangle_{\textup{Nek}}: \widetilde{H}^1_f(V)\otimes\widetilde{H}^1_f(V^*)\lra \QQ_p\otimes_{\ZZ_p}\Gamma\ee
 denote Nekov\'a\v{r}'s height pairing; see \S\ref{subsubsec:selmer} below for the definitions of these objects. Let $\langle\,,\,\rangle_{\textup{Nek},\rho}$ denote the compositum
 $$\langle\,,\,\rangle_{\textup{Nek},\rho}: \widetilde{H}^1_f(V)\otimes\widetilde{H}^1_f(V^*)\lra\QQ_p\otimes\Gamma\stackrel{\rho}{\lra}\QQ_p.$$
Let $\frak{z}^{\textup{BK}}_0 \in H^1(\QQ,V^*)$ denote Beilinson-Kato element (whose basic properties are recalled in \S\ref{subsec:katobeilinson} below) and set ${z}^{\textup{BK}}=\textup{res}_p(\frak{z}_0^{\textup{BK}})$ to be the image of $\frak{z}^{\textup{BK}}_0$ under the localization map
$$\textup{res}_p: H^1(\QQ,V^*)\lra H^1(\QQ_p,V^*).$$
As in \S\ref{subsec:compuationr0} below, one may define elements $[-\textup{ord}_p(q_E)^{-1}] \in \widetilde{H}^1_f(V)$ and $[\exp_{\omega_E}^*(z^{\textup{BK}})]\in\widetilde{H}^1_f(V^*)$ of the extended Selmer groups. We are now ready to state our first theorem.
\begin{thma}[Theorem~\ref{thm:heightformular0} below]
$$\frac{d}{ds}L_p(E,s)\Big|_{s=1}=\left\langle[-\textup{ord}_p(q_E)^{-1}], [\exp_{\omega_E}^*(z^{\textup{BK}})] \right \rangle_{\textup{Nek},\rho}.$$
\end{thma}
This computation should be compared to Benois' results in \cite{benoisajm11} and \cite[Proposition 2.2.4]{benois11nearcentral}.

Observe that when $r_{\textup{an}}(E)=1$, the theorem of Greenberg-Stevens shows that the left hand side of the assertion in Theorem A equals $zero$. Kato's reciprocity law in \cite{ka1} shows that $[\exp_{\omega_E}^*(z^{\textup{BK}})]=0$ as well. Hence, Theorem A says nothing particular when $r_{\textup{an}}(E)=1$. In this case, we shall prove Theorem B below.

When $r_{\textup{an}}(E)\leq 1$, a conjecture of Perrin-Riou (labeled by Conjecture~\ref{conj:PR} below) predicts that Kato's class $\frak{z}_0^{\textup{BK}}\in H^1(\QQ,V^*)$ is non-trivial. Shortly after posting the initial version of this article, the author was notified that R. Venerucci has (partially) proved this conjecture in his thesis\symbolfootnote[1]{We thank M. Bertolini for bringing Venerucci's work to our attention.}, by comparing Kato's class to a suitable Heegner point. We assume that the height pairing $\langle\,,\,\rangle_{\textup{Nek}}$ is non-degenerate. Let $\widetilde{\frak{Z}}_{\textup{BK}} \in \widetilde{H}_f^1(V)\cong\widetilde{H}_f^1(V^*)$ (the identification is via the Weil pairing) denote the lift of the normalized Beilinson-Kato element $\lambda_{\textup{BK}}\,{\cdot}\,{\frak{z}}^{\textup{BK}}_0$. The normalization factor $\lambda_{\textup{BK}}$ is defined in Section~\ref{subsec:definelambdaBK} where we also verify that $\lambda_{\textup{BK}}\neq0$ under our running hypothesis. See also Remark~\ref{rem:whatislambdaBK} where we explain the role this factor plays in our work. The splitting 
$$\frak{s}: \,\textup{Sel}_p(E/\QQ)\otimes\QQ_p\lra\widetilde{H}^1_f(X)\,\,\,\,\,\,\,\, (X=V,V^*)$$ which is used to lift $\lambda_{\textup{BK}}\,{\cdot}\,{\frak{z}}^{\textup{BK}}_0$ to $\widetilde{\frak{Z}}_{\textup{BK}}$ is that of \cite[11.4.2]{nek} and we recall its definition in Section~\ref{subsubsec:liftBKtotilde} for the convenience of the reader. Finally, let $\gamma_0 \in \Gamma$ be a fixed generator satisfying $\log_p(\rho_\textup{cyc}(\gamma_0))=p$.
\begin{thmb}[Theorem~\ref{thm:mainr1} below]
Suppose $r_{\textup{an}}(E)=1$, $E(\QQ)[p]=0$  and the height pairing $\langle\,,\,\rangle_{\textup{Nek}}$ is non-degenarate. 
Then,
$$\frac{1}{2}\left(\frac{d^2}{ds^2}(L_p(E,s))\Big|_{s=1}\right)\otimes \gamma_0=-\frac{1}{\lambda_{\textup{BK}}}\cdot\langle\widetilde{\frak{Z}}_{\textup{BK}},\widetilde{\frak{Z}}_{\textup{BK}} \rangle_{\textup{Nek}}\,,$$
where the equality takes place in $\Gamma$.
\end{thmb}
\begin{rem}
The reader might be concerned that the right hand side in Theorem B is independent of the choice of an isomorphism  $\kappa: \Gamma \ra 1+p\ZZ_p$, whereas the choice of the element $\gamma_0\in\Gamma$ relies on   the choice $\kappa=\rho_{\textup{cyc}}$. Note, however, that the definition of $L_p(E,s)$ (c.f., \S\ref{sec:heightformulas} below) also relies on the cyclotomic character $\rho_{\textup{cyc}}$ and the element $\left(\frac{d^2}{ds^2}(L_p(E,s))\big|_{s=1}\right)\otimes \gamma_0$ would remain unchanged if $\rho_{\textup{cyc}}$ was to be replaced by any other isomorphism $\kappa: \Gamma \ra 1+p\ZZ_p$.
\end{rem}

Our strategy to deduce Theorems A and B is rather straightforward once we dig into Nekov\'a\v{r}'s Selmer complex machine and it is more or less identical to what we have implemented in an earlier version (except in that, in order to prove Theorem B we had relied on an erroneous assertion \cite[Proposition~11.5.11]{nek} in a crucial way, which we replace with Corollary~\ref{cor:incaseallbutoneacyclic} proved as part of this article). The idea is basically to integrate the \emph{derivative} of the Coleman map against what might be considered as the \emph{derivative} of the Beilinson-Kato measure (associated to Beilinson-Kato elements) and to recover via this calculation the height of an appropriately normalized lift of the Beilinson-Kato element. In earlier versions of this paper there was also an ambiguity in the choice of $\lambda_{\textup{BK}}$ which we treat in this version in greater detail.


\begin{rem}
\label{rem:whatislambdaBK}
In this Remark we explain the role that $\lambda_{\textup{BK}}$ plays in this work. For $X=V,V^*$, let 
$$\frak{z}_\infty^{\textup{BK}}=\{\frak{z}^{\textup{BK}}_n\}\in \varprojlim H^1({\QQ_n},X)=H^1(\QQ,X\otimes\LL)$$
denote the Beilinson-Kato element whose key properties are outlined in Section~\ref{subsec:katobeilinson} below. Let $\lambda \in \ZZ_p$ be an arbitrary $p$-adic integer 
and let $[(\lambda\cdot\frak{z}_0^{\textup{BK}}, (z_\ell^+),(\mu_\ell))]=\widetilde{z} \in \widetilde{H}^1_f(X)$ be a lift of $\lambda\cdot{\frak{z}}^{\textup{BK}}_0$ under the splitting $\frak{s}$. We refer the reader to Section~\ref{subsubsec:selmer} to clarify our notation in the previous sentence (borrowed from \cite{nek}). Using Corollary~\ref{cor:incaseallbutoneacyclic} (which applies in our setting since the complex $C^\bullet(\textup{Gal}(\overline{\QQ}_\ell/\QQ_\ell),X)$ of continuos cochains is acyclic for $\ell\neq p$) we have
$$\langle \widetilde{z},\widetilde{z} \rangle_{\textup{Nek}}=-\lambda\cdot\langle \mathcal{D}(\frak{z}_\infty^{\textup{BK}}),z_p^+ \rangle_{\textup{Tate}}$$
where $\mathcal{D}(\frak{z}_\infty^{\textup{BK}})\in H^1(\QQ_p,F_p^-V)$ is the (Bockstein-normalized) \emph{derivative} of the Beilinson-Kato element $ \frak{z}_\infty^{\textup{BK}}$ (defined as in Lemma~\ref{lem:bocksteinnormalizedderivative}). Let $\frak{C}_0 \in H^1(\QQ_p,F_p^+V^*)$ denote the \emph{derivative of the Coleman map} given as in  (\ref{eqn:definederivedcoleman}). We choose $\lambda=\lambda_{\textup{BK}}$ in a way that for that choice $\lambda$ we have
\be\label{eqn:reducetodercolemanahainstderbk}\langle \mathcal{D}(\frak{z}_\infty^{\textup{BK}}),z_p^+ \rangle_{\textup{Tate}}=\langle \mathcal{D}(\frak{z}_\infty^{\textup{BK}}),\frak{C}_0 \rangle_{\textup{Tate}}\,.\ee
We then use the local calculations carried out in Section~\ref{subsub:localcomputationsgeneral} below in order to express the quantity in (\ref{eqn:reducetodercolemanahainstderbk}) as, roughly speaking, the derivative of the Coleman integrated against the derivative of the Beilinson-Kato measure, along the lines we have indicated above.
\end{rem}

Theorem B has the following immediate corollary:
\begin{corc}[Corollary~\ref{cor:only} below]
Under the assumptions of Theorem~B the Mazur-Tate-Teitelbaum conjecture \textup{(\ref{eqn:conjmtt1})} holds true.
\end{corc}

\begin{rem}
\label{rem:venerucci}
A result similar to Theorem~B above has been obtained independently by R. Venerucci, see in particular Corollary 12.32 of his thesis~\cite{veneruccithesis}. 

We remark further that Venerucci's expression for the second derivative of $L_p(E,s)$ at $s=1$  in terms of a $2\times2$ regulator fits better with the perspective offered by ($p$-adic) Beilinson conjectures. On the other hand $d^2/ds^2L_p(E,s)\mid_{s=1}$ is expected to be related to $L^\prime(E,1)$ in this particular setting, for which reason we found it desirable to express this quantity in terms of the ($p$-adic) height of a single element, much in the spirit of the classical Gross-Zagier formula.
\end{rem}
Let $A/\QQ$ be an elliptic curve with good ordinary reduction at $p$. When $\textup{ord}_{s=1}\, L(A,s)=1$, one may compare the order of vanishing of the Mazur-Tate-Teitelbaum $p$-adic $L$-function $L_p(A,s)$ to that of the complex Hasse-Weil $L$-function $L(A,s)$ (as in Corollary C), by making use of the results of \cite{schneider85inventiones} and \cite{pr93abvar}, along with the recent proof of Skinner and Urban~\cite{skinnerurbanmainconjpreprint} of Mazur's main conjecture. Note however that this comparison would still require the non-degeneracy of a certain $p$-adic height pairing. Corollary C in this sense extends the results Schneider and Perrin-Riou to the case when the elliptic curve $E$ in question has split multiplicative reduction at $p$ (in which case the $p$-adic $L$-function attached to $E$ possesses an exceptional zero).

We briefly outline the plan of the paper. In Section~\ref{subsubsec:selmer}, we introduce Nekov\'a\v{r}'s Selmer complexes (whose cohomology yields his \emph{extended Selmer groups}) and discuss their relation with various Selmer groups. In Section~\ref{subsubsec:heights}, we recall Nekov\'a\v{r}'s definition of height pairings in great generality. In Section~\ref{subsub:localcomputationsgeneral}, we carry out a local computation with the local Tate pairing (still in great generality) which is essential for the height calculations in Section~\ref{sec:heightformulas}. In Section~\ref{subsec:colemanmapfortate} (resp., in Section~\ref{subsec:katobeilinson}), we define the Coleman map (resp., introduce Beilinson-Kato elements), which are used to define the elements of the extended Selmer groups on which we shall compute Nekov\'a\v{r}'s height pairing (and compare to the derivatives of the $p$-adic $L$-function $L_p(E,s)$). Once these elements are defined, we carry out the height computations in Section~\ref{subsec:compuationr0} in the case $r_{\textup{an}}(E)=0$ and in Section~\ref{subsec:compuationr1} in the case $r_{\textup{an}}(E)=1$.
\section{Generalities on Nekov\'a\v{r}'s Theory of Selmer complexes}
Let $G$ be a profinite group (given the profinite topology) and let $R$ be a complete discrete valuation ring with finite residue field of characteristic $p$. Let $X$ be a free $R$-module of finite type on which $G$ acts continuously. In this section we very briefly review Nekov\'{a}\v{r}'s theory of Selmer complexes and his definition of extended Selmer groups. Although the treatment in this section is far more general than what is needed for the purposes of this paper (e.g., from~\S\ref{subsec:compuationr0} on $K$ will be $\QQ$ and the Galois module $X$ considered below will be  $T$ or $T^*$ (in degree zero)),  it is still much less general than what is covered in~\cite{nek}.

The $G$-module $X$ is admissible in the sense of~\cite[\S3.2]{nek} and we can talk about the complex of \emph{continuous} cochains $C^\bullet(G,X)$ as in~\S3.4 of {loc.cit}. Let $K$ be a number field and for a finite set  $S$ of places of $K$, let  $S_f$ denote the subset of finite places within $S$. We denote by $K_S$ the maximal subextension of $\overline{K}/K$ which is unramified outside $S$ and set $G_{K,S}$ to be the Galois group $\textup{Gal}(K_S/K)$. For all $w \in S_f$, we write $K_w$ for the completion of $K$ at $w$, and $G_w$ for its absolute Galois group. Whenever it is convenient, we will identify $G_w$ with a decomposition subgroup inside $G_K:=\textup{Gal}(\overline{K}/K)$. We will be interested in the cases when $G=G_{K,S}$ or $G=G_w$ and in the former case, $S$ is chosen to contain all primes above $p$, all primes at which $G$ representation $X$ is ramified and all infinite places of $K$.

\subsection{Selmer complexes}
\label{subsubsec:selmer}
Classical Selmer groups are defined as a subgroup of elements of the global cohomology group $H^1(G_{K,S},X)$ satisfying certain local conditions; see~\cite[\S2.1]{mr02} for the most general definition. The main idea of~\cite{nek} is to impose local conditions on the level of complexes. We go over basics of Nekov\'{a}\v{r}'s theory, for details see~\cite{nek}.

\begin{define}
\label{def:local condition}
\emph{Local conditions} on $X$ are given by a collection $\Delta(X)=\{\Delta_w(X)\}_{w\in S_f}$, where $\Delta_w(X)$ stands for a morphism of complexes of $R$-modules
$$i_w^+(X): U_w^+\lra C^\bullet(G_w,X)$$
 for each $w\in S_f$.
\end{define}

Also set
$$U_v^-(X)=\textup{Cone}\left( U_v^+(X) \stackrel{-i_v^+}{\lra} C^\bullet(G_v,X) \right)$$
and
$$U_S^{\pm}(X)=\bigoplus_{w\in S_f} U_w^{\pm}(X); \,\,\,\,i_S^+(X)= (i_w^+(X))_{w\in S_f}.$$
We also define
$$\textup{res}_{S_f}: C^\bullet(G_{K,S},X) \lra \bigoplus_{w\in S_f}C^\bullet(G_w,X)$$ as the canonical restriction morphism.
\begin{define}
\label{def: selmer complex}
The \emph{Selmer complex} associated with the choice of local conditions $\Delta(X)$ on $X$ is given by the complex
$$\xymatrix@C=.27in{
\widetilde{C}_f^\bullet(G_{K,S},X,\Delta(X)):= \textup{Cone}(C^\bullet(G_{K,S},X)\bigoplus U_S^+(X)\ar[rr]^(.65){\textup{res}_{S_f}-i_{S}^+(X)}&&\bigoplus_{w\in S_f} C^\bullet(G_w,X))[-1]
}$$
 where $[n]$ denotes a shift by $n$. The corresponding object in the derived category will be denoted by $\widetilde{\mathbf{R}\Gamma}_f(G_{K,S},X,\Delta(X))$ and its cohomology by $\widetilde{H}^i_f(G_{K,S},X,\Delta(X))$ (or simply by $\widetilde{H}^i_f(K,X)$ or by $\widetilde{H}^i_f(X)$ when there is no danger of confusion). The $R$-module $\widetilde{H}^1_f(X)$ will be called the \emph{extended Selmer group}.

The object in the derived category corresponding to the complex $C^\bullet(G_{K,S},X)$ will be denoted by ${\mathbf{R}\Gamma}(G_{K,S},X)$.
\end{define}

\subsubsection{Comparison with classical Selmer groups}
\label{subsubsec:comparison}
For each $w \in S_f$, suppose that we are given a submodule
$$H^1_\FF(K_w,X)\subset H^1(K_w,X).$$
The data which  $\FF$ encodes is called a \emph{Selmer structure} on $M$. Starting with $\FF$, one defines the Selmer group as
$$H^1_\FF(K,X):=\ker\left\{H^1(G_{K,S},X)\lra \bigoplus_{w\in S_f}\frac{H^1(K_w,X)}{H^1_\FF(K_w,X)}\right\}.$$

On the other hand,  as explained in~\cite[\S6.1.3.1-2]{nek}, there is an exact triangle
$$
U_S^-(X)[-1]\lra \widetilde{\mathbf{R}\Gamma}_f(G_{K,S},X,\Delta(X)) \lra {\mathbf{R}\Gamma}(G_{K,S},X)\lra U_S^-(X)
$$
which  gives rise to the following exact sequence in the level of cohomology that is used to compare Nekov\'{a}\v{r}'s extended Selmer groups  to classical Selmer groups.
\begin{prop}[\textup{\cite[\S0.8.0 and \S9.6]{nek}}]\label{prop:compare1}
For each $i$, the following sequence is exact:
$$\dots\lra H^{i-1}(U_S^-(X))\lra \widetilde{H}^i_f(X) \lra H^i(G_{K,S},X)\lra H^i(U_S^-(X))\lra \dots$$
\end{prop}

When Nekov\'a\v{r}'s Selmer complex is given by a choice of \emph{Greenberg local conditions}, the associated extended Selmer group compares to an appropriately defined \emph{Greenberg Selmer groups}), whose definitions we now recall. For further details see~\cite{g1, gr, nek}. Let ${I}_w$ denote the inertia subgroup of $G_w$. Suppose we are given an $R[[G_w]]$-submodule $F_w^+X$ of $X$ for each place $w|p$ of $K$, set $F^-_wX=X/F^-_wX$. Then Greenberg's local conditions (in the sense of~\cite[\S6 and \S7]{nek}) are given by
$$U_w^+(X)=\left\{
\begin{array}{cl}
  C^\bullet(G_w,F_w^+X)& \hbox{ if } w|p,    \\\\
  C^\bullet(G_w/I_w,X^{I_w})& \hbox{ if } w\nmid p
\end{array}
\right.$$
with the obvious choice of morphisms
$$i_w^+(X): U_w^+(X)\lra C^\bullet(G_w,X).$$
As in Definition~\ref{def: selmer complex}, we then obtain a Selmer complex and an extended Selmer group, which we denote by $\widetilde{H}^1_f(X)$. Greenberg's local conditions are the only type of local conditions we will deal with from now on.

We now define the relevant Greenberg Selmer structure 
$\FF$ on $M$:
\begin{define}
\label{def:can-selmer}
The \emph{strict Greenberg Selmer structure} $\FF$ is given by
$$H^1_{\FF}(K_w,X)=\left\{
\begin{array}{cl}
   \textup{im}\left(H^1(G_w,F_w^+X)\ra H^1(G_w,X)\right)& \hbox{ \,\,\,\,if } w|p,    \\\\
\ker \left(H^1(G_w,X)\ra H^1(I_w,X) \right)& \hbox{\,\,\,\, if } w\nmid p.
\end{array}
\right.
$$
\end{define}

\begin{rem}
\label{rem:cohomatellvanishes}
When $X=V$, it follows from \cite[Corollary 3.3(i)]{r00} and the proof of \cite[Proposition 6.7]{r00} that $H^1_{\FF}(K_w,V)=0$ for every $w\nmid p$.
\end{rem}

Associated to the Selmer structure $\FF$, we have the following Selmer group (which is called the \emph{strict Selmer group} in~\cite[\S9.6.1]{nek} and denoted by $S_X^{\textup{str}}(K)$):
\be
\label{eqn:selmer group}
H^1_{\FF}(K,X)=\ker\left(H^1(G_{K,S},X)\lra \bigoplus_{w|p}H^1(G_w,F_w^-X)\oplus \bigoplus_{w\nmid p} H^1(I_w,X)   \right).
\ee

Proposition~\ref{prop:compare1} implies directly that:
\begin{prop}
\label{prop:compare}
The following sequence is exact:
$$H^0(G_{K,S},X) \lra \bigoplus_{w|p}H^0(G_w, F_w^-X) \lra \widetilde{H}^1_f(X)\lra H^1_{\FF}(K,X)\lra 0.$$
\end{prop}
When the coefficient ring $R$ is an integral domain, we let $F$ to be its field of fractions. Set $X_F=X\otimes F$ and $F_w^\pm X_F= (F_w^\pm X)\otimes F$. The true Selmer group $\textup{Sel}(K,X)$ is defined as
$$\textup{Sel}(K,X)=\ker\left(H^1(G_{K,S},X)\lra \bigoplus_{w|p}H^1(I_w,F_w^-X_F)\oplus \bigoplus_{w\nmid p} H^1(I_w,X_F)  \right).$$
We also define $H^1_{\FF}(K,X_F)=H^1_{\FF}(K,X)\otimes F$ and $\textup{Sel}(K,X_F)=\textup{Sel}(K,X)\otimes F$.
\begin{rem}
\label{rem:comparestrictselmertoselmer} Note that in case $H^0(G_w, F_w^-X)=0$ for all $w|p$, then the extended Selmer group $\widetilde{H}^1_f(X)$ coincides with the Selmer group $H^1_{\FF}(K,X)$. However, if some $H^0(G_w, F_w^-X)\neq 0$ then $\widetilde{H}^1_f(X)$ is strictly larger than $H^1_{\FF}(K,X)$ (under the assumption that $X^{G_K}$=0, say). This is the main feature of Nekov\'{a}\v{r}'s Selmer complexes: They reflect the existence of exceptional zeros, unlike classical Selmer groups.
\end{rem}

\begin{rem}
\label{rem:compareXT}
In this remark, let $X=T$, $X_F=V$  and $K=\QQ$. It is well-known (c.f., \cite{coatesgreenberg, greenbergiwthell}) that the Selmer group $H^1_{\FF}(\QQ,T)$ compares to the true Selmer group $\textup{Sel}_p(E/\QQ)=\textup{Sel}(\QQ,T)$ by the following exact sequence:
$$0\lra H^1_{\FF}(\QQ,T)\lra \textup{Sel}_p(E/\QQ) \lra H^1(G_p,F_p^-T)_{\textup{tor}}\oplus {\left(\displaystyle{\bigoplus_{\ell \in S_f-\{p\}}}\frak{t}_\ell\right)}$$
where $\frak{t}_\ell={{\ker(H^1(G_\ell,T)}\ra H^1(I_\ell,V))}\big{/}{ \ker(H^1(G_\ell,T)\ra H^1(I_\ell,T))}$. In our setting, the $\ZZ_p$-module $H^1(G_p,F_p^-T)=\textup{Hom}(G_p,\ZZ_p)$ is torsion free and the order of $\frak{t}_\ell$ equals the $p$-part of the Tamagawa factor at $\ell$. We therefore conclude at once that $H^1_{\FF}(\QQ,T)$ is a subgroup of  $\textup{Sel}_p(E/\QQ)$ of finite index, and further infer that:

$\bullet$ $H^1_{\FF}(\QQ,T)=\textup{Sel}_p(E/\QQ)$ if
\begin{itemize}
\item[(i)] $p$ is prime to all Tamagawa factors of $E$ or if,
\item[(ii)] $\textup{Sel}_p(E/\QQ)=0$.
\end{itemize}

$\bullet$ In general, $H^1_{\FF}(\QQ,V)=\textup{Sel}_p(E/\QQ)\otimes \QQ_p$\,.
\end{rem}
\subsection{Height pairings}
\label{subsubsec:heights}
We now recall Nekov\'{a}\v{r}'s definition of height pairings on his extended Selmer groups. All the references in this section are to~\cite[\S11]{nek} unless otherwise stated. Until the end, we assume that $K=\QQ$.

Let $X^*=\textup{Hom}(X,R)(1)$ (in Nekov\'{a}\v{r}'s language this is $\mathcal{D}(X)(1)$, the Grothendieck dual of $X$) and $X_F^*=\textup{Hom}(X_F,F)(1)$. Let  $\Gamma$ be the Galois group $\textup{Gal}(\QQ_{\infty}/\QQ)$. Nekov\'a\v{r}'s height pairing
$$\xymatrix{
\langle\,,\,\rangle_{\textup{Nek}}:\,\, \widetilde{H}^1_f(X) \otimes_{R}\widetilde{H}^1_f(X^*)\ar[r]&R \otimes_{\ZZ_p}\Gamma
}$$
 is defined in two steps:
\begin{itemize}
\item[(i)] Apply the \emph{Bockstein} morphism
$$\xymatrix{\beta:\widetilde{\mathbf{R}\Gamma}_f(X) \ar[r]& \widetilde{\mathbf{R}\Gamma}_f(X)[1]\otimes_{\ZZ_p}\Gamma}$$
  See~\cite[\S11.1.3]{nek} and Appendix A for a precise definition of $\beta$. Let $\beta^1$ denote the map induced on the level of cohomology:
$$\beta^1:\,\, \widetilde{H}_f^1(X) \lra \widetilde{H}^2_f(X)\otimes_{\ZZ_p}\Gamma.$$
\item[(ii)] Use the \emph{Poitou-Tate global duality} pairing
$$\langle\,,\,\rangle_{\textup{PT}}:\,\, \widetilde{H}^2_f(X)\otimes_{R}\widetilde{H}^1_f(X^*) \lra R$$
 on the image of $\beta^1$ inside of $\widetilde{H}^2_f(X)\otimes\Gamma$. Here the global pairing comes from summing up the invariants of the local cup product pairing, see~\cite[\S6.3]{nek} and Definition~\ref{define:globalcupproducpairing} below for more details.
\end{itemize}
Any choice of a homomorphism $\kappa:\Gamma\ra F$ induces an $F$-valued height pairing
$$\langle\,,\,\rangle_{\textup{Nek},\kappa}:\,\, \widetilde{H}^1_f(X_F) \otimes_{R}\widetilde{H}^1_f(X_F^*) \lra F\,\,.$$

\subsection{Computations with the local Tate pairing}
\label{subsub:localcomputationsgeneral}

For $X$ and $X^*$ as above, we set $K=\QQ$ and let $\langle\,,\,\rangle_{\textup{Tate}}: H^1(\Phi_n,X) \otimes  H^1(\Phi_n,X^*) \ra R$
denote the local Tate-pairing. Fix elements $\xi=\{\xi_n\} \in \varprojlim H^1(\Phi_n,X)$ and $\mathbf{z}=\{z_n\} \in \varprojlim H^1(\Phi_n,X^*(1))$ and define
$$\al_{\xi}^{(n)}=\sum_{\tau\in\Gamma_n}\langle \xi_n,z_n^\tau \rangle_{\textup{Tate}}\cdot\tau \in R[\Gamma_n]\,.$$
The elements $\al_{\xi}^{(n)}$ are compatible with respect to restriction maps $R[\Gamma_n]\ra R[\Gamma_m]$ for $m\geq n$ and we may therefore define $\al_\xi=\lim \al_{\xi}^{(n)} \in R[[\Gamma]]$.
\begin{define}
\label{define:derofLxi} Suppose $\xi_0=0$. In this case, we define
\begin{align*}
\textup{Der}_{\rho_{\textup{cyc}}}(\al_\xi)(z_0)&:=\lim_{n\ra\infty}\sum_{\tau\in\Gamma_n} \log_p(\rho_{\textup{cyc}}(\tau^{-1}))\cdot\langle\xi_n^{\tau},z_n\rangle_{\textup{Tate}}\\
&=-\lim_{n\ra\infty}\sum_{\tau\in\Gamma_n} \log_p(\rho_{\textup{cyc}}(\tau))\cdot\langle\xi_n^{\tau},z_n\rangle_{\textup{Tate}}.
\end{align*}
\end{define}
Here we make sense of $\rho_{\textup{cyc}}(\tau)$ as follows for $\tau \in \Gamma_n$. Choose any lift $\tilde{\tau}\in\Gamma$ of $\tau$ and set $\rho_{\textup{cyc}}(\tau)=\rho_{\textup{cyc}}(\tilde{\tau})$. The value of  $\log_p(\rho_{\textup{cyc}}(\tau))$ is therefore well-defined modulo $p^n$, but the limit above clearly does not depend on the choice of lifts $\tilde{\tau}$. See \cite[Lemma 5.9]{kbbheights} for a proof that this limit exists.
\begin{lemma}
\label{lem:derexists}
Suppose $\xi_0=0$. There is an element $\xi^\prime =\{\xi^\prime_n\} \in \varprojlim H^1(\Phi_n,X)$ such  that $\xi=\frac{(\gamma-1)}{\log_p(\rho_{\textup{cyc}}(\gamma))}\cdot\xi^\prime$. Furthermore, $\xi^\prime$ is uniquely determined when the $\LL$-module $\varprojlim H^1(\Phi_n,X)$ has no $(\gamma-1)$-torsion.
\end{lemma}
\begin{proof}
This follows at once from the exactness of the sequence
$$0\lra H^1(\QQ_p,X\otimes\LL)[\gamma-1]\lra H^1(\QQ_p,X\otimes\LL) \stackrel{\gamma-1}{\lra}H^1(\QQ_p,X\otimes\LL) \lra H^1(\QQ_p,T)$$
and using the identification $\varprojlim H^1(\Phi_n,X)=H^1(\QQ_p,X\otimes\LL)$. Here $H^1(\QQ_p,X\otimes\LL)[\gamma-1]$ stands for the $(\gamma-1)$-torsion submodule of $H^1(\QQ_p,X\otimes\LL)$.
\end{proof}
 Note that $\xi_0^\prime$ does not depend on the choice of $\gamma$. 
 \begin{lemma}
\label{lem:calculateder}
Suppose $\xi_0=0$ and let $\xi^{\prime}=\{\xi_n^\prime\}$ is \emph{any} element whose existence was proved in Lemma~\ref{lem:derexists}. Then $\langle \xi^\prime_0,z_0\rangle_{\textup{Tate}}=\textup{Der}_{\rho_{\textup{cyc}}}(\al_\xi)(z_0)\,.$
\end{lemma}
\begin{proof}
Observe that
\begin{align*}
\log_p(\rho_{\textup{cyc}}(\gamma))\sum_{\tau\in\Gamma_n} \log_p(\rho_{\textup{cyc}}(\tau^{-1}))\cdot\xi_n^{\tau}&=\sum_{\tau\in\Gamma_n} \log_p(\rho_{\textup{cyc}}(\tau^{-1}))\cdot(\xi_n^{\prime})^{\tau(\gamma-1)}\\
&=\sum_{\tau\in\Gamma_n}\left( \log_p(\rho_{\textup{cyc}}(\tau^{-1}))(\xi_n^{\prime})^{\tau\gamma}-\log_p(\rho_{\textup{cyc}}(\tau^{-1}))(\xi_n^{\prime})^{\tau}\right)\\
&=\sum_{\sigma\in\Gamma_n}\left( \log_p(\rho_{\textup{cyc}}(\sigma^{-1}))(\xi_n^{\prime})^{\sigma}+\log_p(\rho_{\textup{cyc}}(\gamma))(\xi_n^{\prime})^{\sigma}\right)\\
&\,\,\,\,\,\,\,\,\,\,\,\,\,\,\,\,\,\,\,\,\,\,\,\,\,\,\,\,\,\,\,\,\,\,\,\,\,\,\,\,\,\,\,\,-\sum_{\tau \in \Gamma_n}\log_p(\rho_{\textup{cyc}}(\tau^{-1}))(\xi_n^{\prime})^{\tau}\\
&=\log_p(\rho_{\textup{cyc}}(\gamma))\sum_{\sigma\in\Gamma_n}(\xi_n^{\prime})^{\sigma},
\end{align*}
where all the equalities take place in $R/p^nR$, and the third equality is obtained by setting $\sigma=\tau\gamma$. 
This shows that ${\displaystyle \sum_{\tau\in\Gamma_n} \log_p(\rho_{\textup{cyc}}(\tau^{-1}))\cdot\xi_n^{\tau}=\sum_{\sigma\in\Gamma_n}(\xi_n^{\prime})^{\sigma}}$ (in $R/p^{n-1}R$). By the commutativity of the diagram
$$\xymatrix{H^1(\Phi_n,X) & \times & H^1(\Phi_n,X^*)\ar[d]^{cor}\ar[r]^(.68){\langle,\rangle_{\textup{Tate}}}&R\\
H^1(\QQ_p,X)\ar[u]^{\textup{res}} & \times & H^1(\QQ_p,X^*)\ar[r]^(.68){\langle,\rangle_{\textup{Tate}}}&R
}$$
and the fact that both $\{\xi_n^\prime\}$ and $\{z_n\}$ are norm-coherent, 
we conclude that
$$\left\langle\sum_{\tau\in\Gamma_n} \log_p(\rho_{\textup{cyc}}(\tau^{-1}))\cdot\xi_n^{\tau}\,,\, z_n\right\rangle_{\textup{Tate}}=\langle\xi^\prime_0,z_0\rangle_{\textup{Tate}}$$
in $R/p^{n-1}R$. Proof of the Lemma follows by letting $n\ra \infty$.
\end{proof}
\begin{define}
\label{def:derivedmeasure}
Suppose $\xi_0=0$ and let $\xi^{\prime}=\{\xi_n^\prime\}$  be as above. Define the \emph{derivative of the measure} $\al_{\xi}$ by setting
$$\al_{\xi}^{\prime}:=\al_{\xi^{\prime}}=\left\{\sum_{\tau \in \Gamma_n}\langle\xi_n^{\prime},z_n ^\tau\rangle_{\textup{Tate}}\cdot\tau\right\}\in \LL.$$
Observe that the derived measure $\al^\prime_\xi$ depends both on the choice of $\gamma$ and the choice of $\xi^\prime$.
\end{define}
Let $J=\ker(\LL\ra\ZZ_p)$ denote the augmentation ideal. We have an isomorphism
$$R \otimes_{\ZZ_p} J/J^2 \stackrel{\sim}{\lra}R\otimes_{\ZZ_p}\Gamma\stackrel{\sim}{\lra} R$$
given by
$1\otimes (\gamma-1 \mod J^2) \mapsto \frac{1}{p}\log_p(\rho_{\textup{cyc}}(\gamma))$. Let $1\otimes (\gamma_0-1) \in J/J^2$ denote the image of $1\in R$ under the inverse of this composition.
\begin{lemma}
\label{lem:normalizedderivedmeasure}
$\displaystyle{\frac{(\gamma-1)}{\log_p(\rho_{\textup{cyc}}(\gamma))}}\al_{\xi}^{\prime}\equiv \al_{\xi} \mod J^2$.
\end{lemma}
\begin{proof}
The proof of this is identical to the proof of Lemma~\ref{lem:calculateder}.
\end{proof}
Even when $H^0(\QQ_p,X)=0$ we may define the \emph{derivative} of $\xi$ as follows. Consider the sequence 
\be\label{eqn:bocksteinseq}0\lra X\otimes \Gamma \stackrel{j}{\lra} X\otimes \LL/J^2 {\lra} X\otimes J/J^2\lra 0\ee
where $j$ stands for the map induced from multiplication by $(\gamma-1)/\log_p(\rho_{\textup{cyc}}(\gamma))$. The exact sequence (\ref{eqn:bocksteinseq}) yields the first row of the following commutative diagram with exact rows:
$$\xymatrix@C=1.5pc{H^0(\QQ_p,X)\otimes J/J^2\ar[r]&H^1(\QQ_p,X)\otimes\Gamma \ar[r]^(.47){j}&H^1(\QQ_p,X\otimes\LL/J^2)\ar[r]& H^1(\QQ_p,X)\otimes J/J^2\\
&H^1(\QQ_p,X\otimes\LL)\ar[u]\ar[r]^{j}&H^1(\QQ_p,X\otimes \LL)\ar[u]^{\textup{pr}}\ar[r]&H^1(\QQ_p,X)\ar[u]
}$$
 When $\xi=\{\xi_n\}$ satisfies $\xi_0=0$, it follows from Lemma~\ref{lem:derexists}\footnote{See also the detailed discussion in the appendix regarding this matter. In fact, up to an element of $H^0(\QQ_p,X)$, the element $\mathcal{D}(\xi)$ here is the class denoted by $[Dx_{\textup{Iw}}]$ in Lemma~\ref{lem:bocksteinnormalizedderivative}.} that there exists an element $\mathcal{D}(\xi) \in  H^1(\QQ_p,X)\otimes\Gamma$ (which is in general determined only up to an element of $H^0(\QQ_p,X)\otimes J/J^2$) such that $j(\mathcal{D}(\xi))=\textup{pr}(\xi_\infty)$. In case $H^1(\QQ_p,X\otimes\LL)[\gamma-1]=0$ this element is uniquely determined and in fact relates to the element $\xi_0^\prime$ defined as in Lemma~\ref{lem:calculateder} via $\mathcal{D}(\xi)=\xi_0^\prime\otimes \gamma_0\,.$
Furthermore Lemma~\ref{lem:calculateder} shows for a universal norm $z_0$ that
\be\label{eqn:derivativescompared}\langle\mathcal{D}(\xi),z_0\rangle_{\textup{Tate}}=\textup{Der}_{\rho_{\textup{cyc}}}(\al_\xi)(z_0) \otimes \gamma_0 \in R\otimes\Gamma\,.\ee
This tells us that even though the element $\mathcal{D}(\xi) \in H^1(\QQ_p,X)\otimes\Gamma$ is not uniquely determined, its value on a universal norm $z_0\in H^1(\QQ_p,X^*)$ is.


\section{Height formulas}
\label{sec:heightformulas}
Fix a generator $\{\zeta_{p^n}\}$ of  $\ZZ_p(1)=\varprojlim_n \pmb{\mu}_{p^n}$. Let $E/\QQ$ be an elliptic curve that has split multiplicative reduction at $p$. Then $E$ is a Tate curve at $p$, i.e., it admits a uniformization
$$\mathbb{C}_p^\times/q_E^{\ZZ}\stackrel{\sim}{\lra} E(\mathbb{C}_p)$$
for some $q_E\in \QQ_p^\times$.
The following theorem that was formerly known as Manin's conjecture was proved in \cite{sainetienne}:

\begin{thm}[\emph{Saint-Etienne Theorem}]
\label{thm:saintetienne}
$\log_p(q_E)\neq0$.
\end{thm}

Let $L(E/\QQ,s)$ denote the Hasse-Weil $L$-function attached to $E$. It is known thanks to \cite{wiles,5authors} that $L(E/\QQ,s)$ is an entire function, let $r_{\textup{an}}(E):=\textup{ord}_{s=1} L(E/\QQ,s)$ be the order of vanishing at $s=1$.

Attached to $E$, there is an element $\al_E \in \LL$ (the \emph{Mazur-Tate-Teitelbaum $p$-adic $L$-function}) constructed in \cite{mtt} and characterized by the interpolation formula
$$\chi(\al_E)=\tau(\chi)\frac{L(E,{\chi}^{-1},1)}{\Omega_E^+}$$
for every non-trivial character $\chi$ of $\Gamma$ of finite order, where $\tau(\chi)=\sum_{\delta\in \Delta_n}\chi(\delta)\zeta_{p^{n+1}}^{\delta}$ is the Gauss sum and where $n$ is the smallest integer such that $\chi$ factors through $\Delta_n:=\Gamma/\Gamma^{p^n}$. Furthermore, the Mazur-Tate-Teitelbaum's $p$-adic $L$-function vanishes at the trivial character $\pmb{1}$, namely, $\pmb{1}(\al_E)=0$. Setting
$$L_p(E,s):= \rho_{\textup{cyc}}^{s-1}(\al_E)\,,$$
 we conclude in this case that $L_p(E,1)=0$. A theorem of Greenberg-Stevens \cite{greenbergstevens} expresses the derivative of the $p$-adic $L$-function $L_p(E,s)$ at $s=1$ in terms of the $L$-value:
\be\label{eqn:grste}\frac{d}{ds}L_p(E,s)\Big|_{s=1}=\frac{\log_p(q_E)}{\textup{ord}_p(q_E)}{L(E,1)}/{\Omega_E^+}.\ee
We therefore conclude when $r_{\textup{an}}(E)=0$ or $1$, the order of vanishing of $L_p(E,s)$ at $s=1$ is at least $1+r_{\textup{an}}(E)$. Our goal is to express $\frac{d}{ds}L_p(E,s)\big|_{s=1}$ (resp., $\frac{d^2}{ds^2}L_p(E,s)\big|_{s=1}$) when $r_{\textup{an}}(E)=0$ (resp., when $r_{\textup{an}}(E)=1$) in terms of Nekov\'a\v{r}'s height pairings evaluated  on elements obtained from the Beilinson-Kato elements and the Coleman map, whose basic properties we outline below.

\begin{rem}
\label{rem:powerseriesvsmeasures}
By a slight abuse, we will denote the measure on $\Gamma$ associated to an element $\al \in \LL$ also by $\al$. Then for any continuous character $\psi:\Gamma\ra\mathbb{C}_p$, we will have $\int_\Gamma \psi\cdot d\al=\psi(\al)$. For example, we will sometimes prefer to write $L_p(E,s)=\int_\Gamma\rho_{\textup{cyc}}^{s-1}\cdot d\al_E
$.

\end{rem}
\subsection{The (explicit) Coleman map for a Tate Curve}
\label{subsec:colemanmapfortate}
We review here the definition of the Coleman map following~\cite{rubin96durham} and \cite[Section 8]{kobayashi03}. Let $\frak{O}_n$ denote the ring of integers of $\Phi_n$ and let $\frak{m}_n$ denote the maximal ideal of $\frak{O}_n$ and $\pi_n\in\frak{m}_n$ a fixed uniformizer. Denote $1$-units of $\frak{O}_n$ by $U_{n}^1$. For a fixed generator $\{\zeta_{p^n}\}$ of $\ZZ_p(1)$, one constructs elements $c_n\in \widehat{\mathbb{G}}_m(\frak{m}_n)$ so that the elements $d_n:=1+c_n\in U^1_n$ are norm compatible as $n$ varies and $d_n$ generates $(U_n^1)^{\textbf{N}=1}$ where $\textbf{N}$ stands for the absolute norm from $\Phi_n$ to $\QQ_p$. Let
$$d_\infty=\{d_n\}\in\varprojlim \Phi_n^{\times}\,\widehat{\otimes}\, \ZZ_p \cong\varprojlim H^1(\Phi_n,\ZZ_p(1)) \cong H^1(\QQ_p,\ZZ_p(1)\otimes\LL),$$
where the first isomorphism hollows from Kummer theory and second from~\cite[Proposition II.1.1]{Colmez98Annals}. As $\textbf{N}(d_n)=1$ by construction, it follows that $d_\infty$ is in the kernel of the augmentation map:
$$d_\infty \in \ker(H^1(\QQ_p,\ZZ_p(1)\otimes\LL) \lra H^1(\QQ_p,\ZZ_p(1)))=(\gamma-1)H^1(\QQ_p,\ZZ_p(1)\otimes\LL).$$
Let
\be\label{eqn:definederivedcoleman}\frak{C}_\infty=\{\frak{C}_n\} \in H^1(\QQ_p,\ZZ_p(1)\otimes\LL)=\varprojlim  \Phi_n^{\times}\,\widehat{\otimes}\, \ZZ_p\ee 
be the element chosen such that
$$d_\infty=\frac{(\gamma-1)}{\log_p(\rho_{\textup{cyc}}(\gamma))}\cdot \frak{C}_\infty\,\,.$$ 
It is straightforward to verify that the element $\frak{C}_0$ does not depend on the choice of $\gamma$.
As we have assumed the elliptic curve $E$ has split multiplicative reduction mod $p$, it follows that $E$ is locally a Tate curve, namely that $E_{_{/\QQ_p}}=E_q$ where
$$E_q: y^2+xy=x^3+a_4(q)x+a_6(q)$$
with $q=q_E\in \QQ_p^\times$ satisfying $\textup{ord}_p(q)>0$ and
$$a_4(q)=-\sum_{n\geq1}\frac{n^3q^n}{1-q^n}\,\,\,\,,\,\,\,\, a_6(q)=-\frac{5}{12}{\sum_{n\geq1}\frac{n^3q^n}{1-q^n}+\frac{7}{12}\sum_{n\geq1}\frac{n^5q^n}{1-q^n}}\,\,.$$
Then $E_q$ admits a Tate uniformization
$$\phi: \mathbb{C}_p^\times/q^\ZZ \stackrel{\sim}{\lra} E_q(\mathbb{C}_p).$$
This isomorphism induces an isomorphism of formal groups
$$\widehat{\phi}: \widehat{\mathbb{G}}_m \stackrel{\sim}{\lra}\widehat{E}.$$
Via this isomorphism, we regard the element $c_n \in  \widehat{\mathbb{G}}_m(\frak{m}_n)$ as an element of $\widehat{E}(\frak{m}_n)$, and by the Kummer map also an element of $H^1(\Phi_n,T)$. Using the local duality pairing
$$\langle\,,\, \rangle_{\textup{Tate},E}:  H^1(\Phi_n,T)\times H^1(\Phi_n,T^*)\lra \ZZ_p,$$
we obtain $\ZZ_p[\Gamma_n]$-linear maps
\begin{align*}
\textup{Col}_n: H^1(\Phi_n,T^*)&\ra \ZZ_p[\Gamma_n]\\
\notag z &\mapsto \sum_{\tau \in \Gamma_n}\langle c_n^\tau,z\rangle_{\textup{Tate},E}\,\,\cdot\tau
\end{align*}
which are compatible as $n$ varies with respect to corestriction maps and natural projections. Hence these maps yield in the limit a $\LL$-equivariant map
$$\textup{Col}:\varprojlim H^1(\Phi_n,T^*)\cong H^1(\QQ_p,T^*\otimes\LL)\lra \LL.$$
As explained in \cite[\S4]{kobayashi06},
\be
\label{eqn:colemanexplicit}
\textup{Col}_n(z)=\sum_{\tau \in \Gamma} \langle d_n^\tau,\textup{res}_p^s(z_n) \rangle_{\textup{Tate},\mathbb{G}_m}\cdot\tau
\ee
where $\textup{res}_p^s: H^1(\Phi_n,T^*)\ra H^1(\Phi_n,F_p^-T^*)$ is the projection on to the singular quotient and 
$$\langle\, ,\, \rangle_{\textup{Tate},\mathbb{G}_m}: H^1(\Phi_n,F_p^+T)\times H^1(\Phi_n,F_p^-T^*)\lra \ZZ_p$$
is the local Tate pairing for $\mathbb{G}_m$ that compares to $\langle\, ,\, \rangle_{\textup{Tate},E}$ via the commutative diagram
$$\xymatrix@C=.6pc{H^1(\Phi_n,T)&\times& H^1(\Phi_n,T^*)\ar[d]^{i_v^-}\ar[rrr]^(.62){\langle\, ,\, \rangle_{\textup{Tate},E}}&&& \ZZ_p\\
H^1(\Phi_n,F_p^+T)\ar[u]^{i_p^+}&\times& H^1(\Phi_n,F_p^-T^*)\ar[rrr]^(.64){\langle\, ,\, \rangle_{\textup{Tate},\mathbb{G}_m}}&&& \ZZ_p\ar@{=}[u]
}$$
\begin{rem}
\label{rem:normalizationoftheweilpairing}
More precisely, the diagram above looks as follows:
$$\xymatrix@C=.35pc{H^1(\Phi_n,T)&\otimes& H^1(\Phi_n,T^*)\ar[d]^{i_p^-}\ar[rrr]^(.45){\cup}&&& H^2(\Phi_n,T\otimes T^*)\ar[rrr]^(.52){\textup{Weil}}&&&H^2(\Phi_n,\ZZ_p(1))\ar[rrr]^(.71){\textup{inv}_p}&&&\ZZ_p\\
H^1(\Phi_n,\ZZ_p(1))\ar[u]^{i_p^+}&\otimes& H^1(\Phi_n,\ZZ_p)\ar[rrr]^(.4){\cup}&&& H^2(\Phi_n,\ZZ_p(1)\otimes\ZZ_p)\ar[rrr]^(.55){\times}&&&H^2(\Phi_n,\ZZ_p(1))\ar[rrr]_(.72){\textup{inv}_p}\ar@{=}[u]&&&\ZZ_p
}$$
Recall that $i_p^{+}$ is induced from the Tate uniformization and $\times$ is the usual multiplication. Note that both Tate uniformization and Weil pairing is defined up to sign and we (implicitly) make a compatible choice so as to make the diagram above commute. 
\end{rem}
\begin{define} We define the following map (also denoted by $\textup{Col}$) 
$$\textup{Col}:\varprojlim H^1(\Phi_n,F_p^-T^*)\cong H^1(\QQ_p,F_p^-T^*\otimes\LL)\lra \LL$$
obtained from the compatible family of maps $\{\textup{Col}_n\}$ from (\ref{eqn:colemanexplicit}).
\end{define}
\subsection{Beilinson-Kato elements}
\label{subsec:katobeilinson}
Given an elliptic curve $E$, Kato has constructed an element
$$\frak{z}_\infty^{\textup{BK}}=\{\frak{z}^{\textup{BK}}_n\}\in \varprojlim H^1({\QQ_n},T^*)\otimes\QQ_p=H^1(\QQ,T^*\otimes\LL)\otimes\QQ_p$$
which has the property that
\be\label{eqn:katofundamental}\textup{Col}(\textup{res}_p(\frak{z}_\infty^{\textup{BK}}))=\al_E,\ee
where $\textup{res}_p: H^1(\QQ_n,-)\ra H^1(\Phi_n,-)$ is the restriction to $G_p$.
 To ease notation we write $z^{\textup{BK}}_n=\textup{res}_p(\frak{z}_n^{\textup{BK}})$ and write $z^{\textup{BK}}$ in place of $z_0^{\textup{BK}}\in H^1(\QQ_p,T^*)\otimes\QQ_p$. For each $n\geq 0$, let
 $$\textup{res}_p^s: H^1(\Phi_n,T^*)\otimes\QQ_p\lra H^1(\Phi_n,F^-_pT^*)\otimes\QQ_p$$
 denote the map induced from natural projection.
 \begin{rem}
 \label{rem:residuallyirred} It may be proved that Beilinson-Kato elements are locally integral, namely that
 ${z}_n^{\textup{BK}}\in H^1(\Phi_n,T^*).$ In case $E(\QQ)[p]=0$, the Beilinson-Kato elements are globally integral as well: $\frak{z}_\infty^{\textup{BK}}=\{\frak{z}^{\textup{BK}}_n\}\in H^1(\QQ,T^*\otimes\LL).$
 \end{rem}
 In \cite[\S3.3.2]{pr93grenoble} Perrin-Riou proposes the following:
 \begin{conj}
 \label{conj:PR}
 The element $\frak{z}_0^{\textup{BK}}\in H^1(\QQ,T^*)\otimes\QQ_p$ is non-trivial iff $\textup{ord}_{s=1}\,L(E,s)\leq 1$.
 \end{conj}
 In this article, we need the ``if" part of this conjecture and this has been established by Venerucci as part of his thesis work:
 \begin{thm}[Venerucci]
 \label{thm:rodolfopr}
If $\textup{ord}_{s=1}\,L(E,s)\leq 1$ then the element $\frak{z}_0^{\textup{BK}}\in H^1(\QQ,T^*)\otimes\QQ_p$ is non-trivial.
\end{thm}
\subsubsection{Lifting the Selmer group to the extended Selmer group}
\label{subsubsec:liftBKtotilde}
For $X=V,V^*$, recall that we have an exact sequence 
$$0\lra H^0(\QQ_p,F_p^-X) \lra \widetilde{H}^1_f(X)\lra \textup{Sel}_p(E/\QQ)\otimes\QQ_p\lra 0$$
by Proposition~\ref{prop:compare} and Remark~\ref{rem:comparestrictselmertoselmer}. Following \cite[11.4.2]{nek}, this sequence admits a natural splitting
$$\xymatrix@C=.3pc@R=.5pc{\frak{s}:& \textup{Sel}_p(E/\QQ)\otimes\QQ_p \ar[rr]&& \widetilde{H}^1_f(X)\\
&[x]\ar@{|->}[rr]&&[(x,(x_\ell^+),(\mu_\ell))]
}$$
which is, according to loc.cit., characterized by the requirement that 
$$[x_p^+] \in \ZZ_p^\times\widehat{\otimes}\QQ_p\subset \QQ_p^\times\widehat{\otimes}\QQ_p=H^1(\QQ_p,F_p^+X)\,.$$ 
Let us explain this in detail. Suppose that
$$[(x,(x_\ell^+),(\mu_\ell))]\,,\,[(x,(\widetilde{x}_\ell^+),(\widetilde{\mu}_\ell))] \in \widetilde{H}^1_f(X)$$
and $[x_p^+], [\widetilde{x}_p^+] \in \ZZ_p^\times\widehat{\otimes}\QQ_p$\,. For each $\ell \in S$ set $z_\ell^+=x_\ell^+-\widetilde{x}_\ell^+$ and $\lambda_p=\mu_\ell -\widetilde{\mu}_\ell$. We contend to prove that the cocyle
$$(x,(x_\ell^+),(\mu_\ell))-(x,(\widetilde{x}_\ell^+),(\widetilde{\mu}_\ell))=(0,(z_\ell^+),(\lambda_\ell))\in \widetilde{Z}^1_f(X)$$
is in fact a coboundary. 

First of all for $\ell \in S$, $\ell\neq p$ the cocycle $z_\ell^+ \in Z^1(U_\ell^+(X))$ is a coboundary since the complex $U_\ell^+(X)$ acyclic.  This in turn means that $z_\ell^+=d\lambda_\ell^+$ for some $\lambda_{\ell}^+ \in U_\ell^{+}(X)^0$\,. Hence 
$$-d\,i_{\ell}^+(\lambda_\ell^+)=-i_\ell(z_\ell^+)=d\lambda_\ell$$
and $-i_{\ell}^+(\lambda_\ell^+)=\lambda_\ell$ since $H^0(\QQ_\ell,X)=0$.  

Since $(0,(z_\ell^+),(\lambda_\ell))$ is a cocycle we have $i_p^+(z_p^+)=-d\lambda_p$\,, which in particular means that $i_p^+([z_p^+])=0$. 
Considering the commutative diagram 
$$\xymatrix{[z_p^+]\in H^1(\QQ_p,F_p^+X)\ar[r]^(.60){i_p^+}& H^1(\QQ_p,X)\\
\,\,\,\,\,\,\,{\QQ_p^{\times}}\widehat{\otimes}\QQ_p\,\,\,\,\,\ar[r]_(.45){\textup{red}}\ar[u]^\kappa&{\QQ_p^\times}\widehat{\otimes}\QQ_p\Big{/}q_E^{\ZZ}\widehat{\otimes}\QQ_p\ar@{^{(}->}[u]_{\Psi}
}$$
where $\kappa$ is the Kummer isomorphism and $\Psi$ is the Tate uniformization followed by the Kummer map on $E(\QQ_p)$. Our conclusion that $i_p^+([z_p^+])=0$ translates via this diagram to the requirement that 
$$\ZZ_p^{\times}\widehat{\otimes}\QQ_p \ni [z_p^+] \in q_E^{\ZZ}\widehat{\otimes}\QQ_p\,.$$
Since $q_E \in p\ZZ_p$ it follows that $z_p^+=d\lambda_p^+$ for a unique (since $H^0(\QQ_p,F_p^+X)=0$) cochain $\lambda_p^+$. Furthermore,
$$-d\, i_p^+(\lambda_p^+)=-i_p^+(z_p^+)=d\lambda_p$$
and $- i_p^+(\lambda_p^+)=\lambda_p$ since $H^0(\QQ_p,X)=0$.
Now observing that 
$$(0,(z_\ell^+),(\lambda_\ell))=d(0,(\lambda_\ell^+),0)$$
is a coboundary and we conclude that 
$$[(x,(x_\ell^+),(\mu_\ell))]=[(x,(\widetilde{x}_\ell^+),(\widetilde{\mu}_\ell))]$$
as we desired to prove. We may in fact describe this lift even more explicitly. It follows using Saint-Etienne theorem that $\frak{B}=\{p,q_E\}$ is an ordered basis of $\widehat{\QQ}^\times_p{\otimes}\QQ_p$. Write $q_E=p^{\textup{ord}_p(q_E)}\,u_E$\,.

\begin{define}\label{def:alphax} For an element $[x] \in \textup{Sel}_p(E/\QQ)\otimes\QQ_p$, let $\alpha(x)$ denote the first coordinate of $\textup{res}_p([x])$ with respect to this basis. Let $x_p^+$ be any cocycle representing the class
$$[x_p^+]:=u_E\otimes\frac{-\alpha(x)}{\textup{ord}_p(q_E)} \in \widehat{\QQ}^\times_p{\otimes}\QQ_p\,.$$
\end{define}
Note that the class of 
\begin{align*}H^1(\QQ_p,X)\ni[\textup{res}_p(x)-i_p^+(x_p^+)]&=\textup{res}_p([x])-i_p^+([x_p^+])\\
&=\left(p\otimes \alpha(x) +u_E\otimes\frac{\alpha(x)}{\textup{ord}_p(q_E)}\right)\cdot q_E^{\ZZ}\widehat{\otimes}\QQ_p\\
\tag{$\star$}&=\left(u_E\otimes\frac{-\alpha(x)}{\textup{ord}_p(q_E)} +u_E\otimes\frac{\alpha(x)}{\textup{ord}_p(q_E)}\right)\cdot q_E^{\ZZ}\widehat{\otimes}\QQ_p\\
&=0
\end{align*}
vanishes (where the equality ($\star$) holds true since $p \equiv u_E\otimes \frac{-1}{\textup{ord}_p(q_E)}$ mod $q_E^{\ZZ}\widehat{\otimes}\QQ_p$) and therefore 
$$\textup{res}_p(x)-i_p^+(x_p^+)=d\mu_p$$ 
for a unique (as $H^0(\QQ_p,X)=0$) cochain $\mu_p \in C^0(\QQ_p,X)$.
 \subsection{Height formulas in the case $r_{\textup{an}}(E)=0$}
\label{subsec:compuationr0}
\begin{prop}[Kato]
If $L(E,1)\neq 0$ then $\textup{Sel}_p(E/\QQ)$ is finite and $H^1_{\FF}(\QQ,V)=0$.
\end{prop}
In this case the exact sequence of Proposition~\ref{prop:compare} induces isomorphism
\be\label{eqn:isomnek1}
H^0(G_p,F_p^-X)\stackrel{\sim}{\lra} \widetilde{H}^1_f(X)
\ee
for  for $X=V,V^*$. Let  $\alpha \in H^0(G_p,F_p^-V)$ and $\alpha^* \in H^0(G_p,F_p^-V^*)$. Denote their respective images under the isomorphism (\ref{eqn:isomnek1})) by $[\alpha]$ and $[\alpha^*]$. These elements are given explicitly as follows (we explain only for $X=V$ and the class $[\alpha]$): The exact sequence (\ref{eqn:tateexactseq}) yields an injection
\be\label{eqn:defdeltaalpha} \partial_p: H^0(G_p,F_p^-V) \hookrightarrow H^1(G_p, F_p^+V)\ee
which is obtained via the snake lemma applied on the diagram
$$\xymatrix@R=1pc{0\ar[r]&Z^1(G_p,F^+_pV) \ar[r]^(.55){i_p^+}&Z^1(G_p,V) \ar[r]^(.47){\iota_p^{-}}&Z^1(G_p,F^-_pV)&\\ 
0\ar[r]&C^0(G_p,F_p^+V)\ar[r]_(.55){i_p+}\ar[u]^{d}&C^0(G_p,V)\ar[r]_(.45){\iota_p^-}\ar[u]^d&C^0(G_p,F_p^-V)\ar[r]\ar[u]^d&0}$$
Namely, since $\iota_p^-$ on the lower row is surjective, there is an $\widetilde{\alpha} \in C^0(G_p,X)$ such that $\iota_p^-(\widetilde{\alpha})=\alpha$. As $d\alpha=0$, it follows that $d\widetilde{\alpha}=i_p^+(\beta)$ for some (unique) $\beta\in Z^1(G_p,F^+_pX)$. Then $\partial_p(\alpha):=[\beta]$ and we set 
$$[\alpha]:=[(0,\beta,\widetilde{\alpha})]\,.$$
It is easy to see that this class is independent of the choice of $\widetilde{\alpha}$

Let $z:G_\QQ \twoheadrightarrow \Gamma$ be the tautological homomorphism. Letting $G_\QQ$ act trivially on $\Gamma$, one may view $z$ as an element of $H^1(\QQ,\Gamma)=\textup{Hom}(G_\QQ,\Gamma)$. Its restriction $z_p \in H^1(G_p,\Gamma)$ also corresponds to the tautological homomorphism $G_{p}\twoheadrightarrow\Gamma$, where we now view $\Gamma$ as the decomposition group of $p$ inside $\textup{Gal}(\QQ_\infty/\QQ)$.
\begin{prop}
\label{prop:computeheightgenerally}
Let $z_p \cup \alpha_p^* \in H^1(G_p,\QQ_p\otimes\Gamma)= H^1(G_p,\QQ_p)\otimes\Gamma$ be the cup-product of $z_p$ and $\alpha_p^*$\,. Then we have the following equality in $\QQ_p\otimes \Gamma$:
$$\langle[\alpha], [\alpha^*] \rangle_{\textup{Nek}}=\langle \partial_p(\alpha),-z_p\cup \alpha_p^* \rangle_{\textup{Tate}}\,.$$
\end{prop}
\begin{proof}
This follows from \cite[Corollary 11.4.7]{nek}, along with the remark 11.3.5.3 of loc.cit. Note that $[\alpha]:=[(0,\beta,\widetilde{\alpha})]$ and $[\beta]=\partial_p(\alpha)$\,.
\end{proof}
Recall the local Beilinson-Kato element $z^{\textup{BK}}:=\textup{res}_p(\frak{z}_0^{\textup{BK}})\in H^1(\QQ_p,T^*)$ and the element $\frak{C}_0 \in H^1(G_p,F_p^+T)\cong\widehat{\QQ}_p^{\times}$ we have obtained using the explicit description of Coleman map. 
Recall also the homomorphism $\rho:\Gamma \ra \ZZ_p$\,, which is  the compositum of the maps
$$\xymatrix{\rho: \Gamma \ar[r]^(.46){{\rho_{\textup{cyc}}}}&1+p\ZZ_p\ar[rr]^(.55){-E_p(1)^{-1}\log_p}&&\ZZ_p\,,}$$
where $E_p(s)=1-p^{-s}$ is the Euler factor at $p$. Let $\frak{a}_p: \widehat{\QQ}^\times_p\ra G_p^{\textup{ab}}$ denote local Artin reciprocity map (normalized to send uniformizers to geometric Frobenii) and $\kappa:\widehat{\QQ}^\times_p\stackrel{\sim}{\ra} H^1(\QQ_p,\ZZ_p(1))$ the Kummer isomorphism.
\begin{thm}
\label{thm:heightformular0}$\,$
\begin{itemize}
\item[(i)] $\left\langle[1],[1]\right\rangle_{\textup{Nek}}=E_p(1)^{-1}\log_p(u_E)$.
\item[(ii)] For every $[y_f]=[(y,(y_\ell^+),(\mu_\ell))]\in \frak{s}\left(\textup{Sel}_p(E/\QQ)\otimes\QQ_p\right)\subset\widetilde{H}^1_f(V^*)$  
$$\left\langle[1],[y_f]\right\rangle_{\textup{Nek}}=-\ell_p([y_p^+])$$
and
$$\left\langle[1],[y_f]\right\rangle_{\textup{Nek},\rho}=E_p(1)^{-1}\log_p([y_p^+])\,.$$
Here $\ell_p$ is the compositum
$$\ell_p=z_p\circ \frak{a}_p\circ \kappa^{-1} : H^1(\QQ_p,F_p^+T) \twoheadrightarrow \Gamma\,.$$
\item[(ii)] $\left\langle[\textup{ord}_p(q_E)^{-1}], [\exp_{\omega_E}^*(z^{\textup{BK}})] \right \rangle_{\textup{Nek},\rho}=\langle \frak{C}_0, \textup{res}_p^{s}(\frak{z}_0^{\textup{BK}})\rangle_{\textup{Tate}}$\,.
\item[(iii) ]$\frac{d}{ds}L_p(E,s)\Big|_{s=1}=\left\langle[-\textup{ord}_p(q_E)^{-1}], [\exp_{\omega_E}^*(z^{\textup{BK}})] \right \rangle_{\textup{Nek},\rho}$.
\end{itemize}
\end{thm}
\begin{proof}
Both (i) and (ii) follow from \cite[Corollary 11.4.7]{nek}; we give a sketch of the proof for (i). Let $\chi_p$ be the compositum $\chi_p=\rho\circ z_p\circ \frak{a}_p:\widehat{\QQ}_p^\times\ra\ZZ_p\,$.
Since the image of $1 \in\QQ_p= H^0(G_p,F_p^-V)$ under the map (\ref{eqn:defdeltaalpha}) is $q_E$\,, it follows from Proposition~\ref{prop:computeheightgenerally} that
\begin{align}\notag\left\langle[1], [1] \right \rangle_{\textup{Nek},\rho}&=-\left\langle q_E,\chi_p\right\rangle_{\textup{Tate}}\\
\label{eqn:trial11}&=-\left\langle u_{E},\chi_p \right\rangle_{\textup{Tate}}\\
\label{eqn:trial33}&=E_p(1)^{-1}\,\log_p(u_{E})
\end{align}
where the equality (\ref{eqn:trial11}) is because the homomorphism $z_p$ factors through the inertia subgroup of $G_p^{\textup{ab}}$ and (\ref{eqn:trial33}) follows using \cite[Lemma II.1.4.5]{katodr} which asserts that
$$\log_p\circ \rho_{\textup{cyc}}\circ z_p \circ \frak{a}_p: \widehat{\QQ}_p^\times \lra  \widehat{\QQ}_p^\times$$
equals $\log_p$ with our choice of normalizations. This proof of (i) follows. The proof of (ii) may be extracted from \cite[Corollary 11.4.7]{nek} in a similar manner.

It now follows from (i) that 
\begin{align}
\notag \left\langle[\textup{ord}_p(q_E)^{-1}], [\exp_{\omega_E}^*(z^{\textup{BK}})] \right \rangle_{\textup{Nek},\rho}&=E_p(1)^{-1}\textup{ord}_p(q_E)^{-1}\log_p(u_E)\exp_{\omega_E}^*(z^{\textup{BK}})\\
\label{eqn:trial2}&=\langle \frak{C}_0, \textup{res}_p^{s}(\frak{z}_0^{\textup{BK}})\rangle_{\textup{Tate}}
\end{align}
where  (\ref{eqn:trial2}) is the main calculation carried out in \cite[\S4]{kobayashi06}. This proves (iii).

To prove (iv) observe that $\frac{d}{ds}\rho_{\textup{cyc}}^{s-1}=\log_p\rho_{\textup{cyc}}\cdot \rho_{\textup{cyc}}^{s-1}$, hence
\begin{align*}\frac{d}{ds}L_p(E,s)\Big|_{s=1}&=\int_\gamma \log_p\rho_{\textup{cyc}}\cdot d\al_E\\
&=\lim_{n\ra\infty} \sum_{\tau \in \Gamma_n} \log_p\rho_{\textup{cyc}}(\tau)\left\langle d_n^\tau\,,\,\textup{res}_p^s(\frak{z}^{\textup{BK}}_\infty)\right\rangle_{\textup{Tate}}\\
&=\lim_{n\ra\infty} \left\langle \sum_{\tau \in \Gamma_n} \log_p\rho_{\textup{cyc}}(\tau)\cdot d_n^\tau\,,\,\textup{res}_p^s(\frak{z}^{\textup{BK}}_\infty)\right\rangle_{\textup{Tate}}
\end{align*}
where the second equality follows from the explicit description of the Coleman map (essentially (\ref{eqn:colemanexplicit}), see also \cite[p. 572]{kobayashi06}). By Lemma~\ref{lem:calculateder} applied with $X=F^+_pT$, $X^*=F_p^-T^*$, $\xi=d_\infty$ (so that $\xi^\prime_0=\frak{C}_0$) and $\mathbf{z}=\textup{res}_p^s(\frak{z}^{\textup{BK}}_\infty)$,
$$\lim_{n\ra\infty} \left\langle \sum_{\tau \in \Gamma_n} \log_p\rho_{\textup{cyc}}(\tau)\cdot d_n^\tau\,,\,\textup{res}_p^s(\frak{z}^{\textup{BK}}_\infty)\right\rangle_{\textup{Tate}}=\langle \frak{C}_0,\textup{res}_p^s(\frak{z}_0^{\textup{BK}})\rangle_{\textup{Tate}}$$
and (iv) now follows from (iii).
\end{proof}
\subsection{Height formulas in the case $r_{\textup{an}}(E)=1$}
\label{subsec:compuationr1}
Until the end of this article, suppose that $r_{\textup{an}}(E)=1$. Assume in addition that $E(\QQ)[p]=0$. As we have noted in Remark~\ref{rem:residuallyirred}, this assumption implies that Beilinson-Kato elements are integral:
$$\frak{z}_\infty^{\textup{BK}}=\{\frak{z}^{\textup{BK}}_n\}\in H^1(\QQ,T\otimes\LL).$$
Above we had introduced Beilinson-Kato elements $\frak{z}^{\textup{BK}}_0$ as elements of the cohomology group $H^1(\QQ,T^*)$. Using the natural isomorphism $T\cong T^*$ we may regard them as classes for $T$ as well.
Recall that $z^{\textup{BK}}:=\textup{res}_p(\frak{z}^{\textup{BK}}_0)\in H^1(\QQ_p,T)$.
\begin{prop}
\label{prop:locpkatononzero}
Under the running assumptions ${z}^{\textup{BK}}\neq 0$.
\end{prop}
\begin{proof}
Assume on the contrary that
\be\label{eqn:katovanishes}
{z}^{\textup{BK}}=\textup{res}_p(\frak{z}^{\textup{BK}}_0)=0.
\ee
Let $\FF_{\textup{str}}$ denote the Selmer structure on $T$ given by
\begin{itemize}
\item $H_{\FF_{\textup{str}}}(\QQ_\ell,T)=H_{\FF}(\QQ_\ell,T)$, if $\ell\neq p$,
\item $H_{\FF_{\textup{str}}}(\QQ_p,T)=0$.
\end{itemize}
so that (\ref{eqn:katovanishes}) amounts to saying $\frak{z}^{\textup{BK}}_0 \in H_{\FF_{\textup{str}}}(\QQ,T)$. As $\frak{z}^{\textup{BK}}_0$ is non-torsion thanks to our running assumptions and Theorem~\ref{thm:rodolfopr}, it follows that $\textup{rank}_{\ZZ_p}(H^1_{\FF_{\textup{str}}}(\QQ,T))\geq 1.$

Let $\FF_{\textup{str}}$ denote also the propagation of the Selmer structure (in the sense of \cite{mr02}) to $T/p^nT$. For any positive integer $n$, identify the quotient $T/p^nT$ with $E[p^n]$. By \cite[Lemma 3.7.1]{mr02}, we have an injection
$$H_{\FF_{\textup{str}}}(\QQ,T)/p^nH_{\FF_{\textup{str}}}(\QQ,T) \hookrightarrow H_{\FF_{\textup{str}}}(\QQ,T/p^nT)=H_{\FF_{\textup{str}}}(\QQ,E[p^n])$$
induced from the projection $T\ra T/p^nT$. This shows that
\be\label{eqn:lowerboundonstr}
\textup{length}_{\ZZ_p}\left(H_{\FF_{\textup{str}}}(\QQ,E[p^n])\right)\geq n.
\ee
Let now $\FFc$ denote the canonical Selmer structure on $T$, given by
\begin{itemize}
\item $H_{\FF_{\textup{can}}}(\QQ_\ell,T)=H_{\FF}(\QQ_\ell,T)$, if $\ell\neq p$,
\item $H_{\FF_{\textup{can}}}(\QQ_p,T)=H^1(\QQ_p,T)$.
\end{itemize}
Let $\FFc^*$ denote the dual Selmer structure on $\textup{Hom}(T,\pmb{\mu}_{p^\infty})\cong E[p^\infty]$, where the isomorphism is obtained via the Weil-pairing. The propagation of $\FFc^*$ on $E[p^\infty]$ to its submodule $E[p^n]$ will also be denoted by $\FFc^*$. It follows from \cite[Lemma I.3.8(i)]{r00} (together with the discussion in \cite[\S6.2]{mr02}) that we have an inclusion
$$H_{\FF_{\textup{str}}}(\QQ_\ell,E[p^n])\subset H_{\FF_{\textup{can}}^*}(\QQ_\ell,E[p^n])$$
for every $\ell$, which in turn shows that together with (\ref{eqn:lowerboundonstr}) that
\be
\label{eqn:lowerboundoncan}
\textup{length}_{\ZZ_p}\left(H_{\FF_{\textup{can}}^*}(\QQ,E[p^n])\right)\geq n.
\ee
On the other hand, as $\frak{z}^{\textup{BK}}_0\neq0$, it follows from \cite[Cor. 5.2.13]{mr02} that $H_{\FF_{\textup{can}}^*}(\QQ,E[p^\infty])$ is finite. This however shows that the length of
$$H_{\FF_{\textup{can}}^*}(\QQ,E[p^n])\cong H_{\FF_{\textup{can}}^*}(\QQ,E[p^\infty])[p^n]$$
(where the isomorphism is thanks to \cite[Lemma 3.5.3]{mr02}, which holds true here owing to our assumption that $E(\QQ)[p]=0$) is bounded independently of $n$. This contradicts (\ref{eqn:lowerboundoncan}) and shows that our assumption (\ref{eqn:katovanishes}) is wrong.
\end{proof}
\begin{rem}
 \label{lem:prreduction}
 In this remark we elaborate on the ``only if'' part of Conjecture~\ref{conj:PR}. Suppose $\frak{z}_0^{\textup{BK}}\in H^1(\QQ,T^*)$ is non-torsion\symbolfootnote[3]{Under the assumption that $E(\QQ)[p]=0$, the $\ZZ_p$-module $H^1(\QQ,T^*)$ is torsion-free. Hence, our assumption amounts to asking that $\frak{z}_0^{\textup{BK}}\neq 0$\,.}.  It follows by the theory of Euler systems that the strict Selmer group
 $$H^1_{\FFc^*}(\QQ,V/T):=\ker(H^1_{\FF}(\QQ,V/T)\lra H^1(\QQ_p,V/T))$$
 is finite. It then follows from global duality (c.f., Theorem 5.2.15 and Corollary 5.2.6 of \cite{mr02}) that
 \be\label{eqn:surelyrankone}\textup{rank}_{\ZZ_p}(H^1_{\FFc}(\QQ,T^*))= \textup{dim}_{\QQ_p}(V^*)^-=1,\ee
where $(V^*)^-$ stands for the $-1$-eigenspace of $V^*$ of a fixed complex conjugation in $G_\QQ$. This in turn shows that $\textup{rank}_{\ZZ_p}(\textup{Sel}(\QQ,T^*))\leq 1$. The conjecture of Birch and Swinnerton-Dyer then predicts the assertion of Conjecture~\ref{conj:PR}.

Suppose now that $\textup{rank}_{\ZZ_p}(\textup{Sel}(\QQ,T^*))= 0$. As explained in \eqref{eqn:surelyrankone}, the $\ZZ_p$-module $H^1_{\FFc}(\QQ,T^*)$ is of rank $1$ and that $\textup{res}_p^s(\frak{z}^{\textup{BK}}_0)\neq 0$. Kato's reciprocity law implies in this case that $L(E,1)\neq 0$, \emph{unconditionally}.

In the case $\textup{rank}_{\ZZ_p}(\textup{Sel}(\QQ,T^*))= 1$, unfortunately we are not able to go this far. As $\textup{rank}_{\ZZ_p}(\textup{Sel}(\QQ,T^*))= 1$, we conclude by \eqref{eqn:surelyrankone} that $H^1_{\FFc}(\QQ,T^*)\otimes\QQ_p=\textup{Sel}(\QQ,T^*)\otimes\QQ_p$ and hence $\frak{z}^{\textup{BK}}_0 \in \textup{Sel}(\QQ,T^*)\otimes\QQ_p$. One would then expect to relate the height of $\frak{z}^{\textup{BK}}_0$ to $L^\prime(E,1)$\symbolfootnote[2]{As a matter of fact, as $\textup{Sel}(\QQ,T^*)$ is rank one, one would expect that $\frak{z}^{\textup{BK}}_0$ relates to Heegner points. This indeed is the content of Perrin-Riou's conjecture.} 
and conclude this way that $L^\prime(E,1)\neq 0$. This, however, seems untractable at this stage\symbolfootnote[4]{See, however, Venerucci's thesis for progress in this direction.}$\,$\symbolfootnote[1]{When the author was preparing this version of this article, Venerucci indeed announced a proof of this conjecture under the additional assumption that the $p$-part of the Tate-Shafarevich group is finite.}. When $p$ is a good-ordinary prime, Perrin-Riou in \cite{pr93grenoble} shows that the $p$-adic height of $\frak{z}^{\textup{BK}}_0$ is related to the derivative of the Mazur-Tate-Teitelbaum $p$-adic $L$-function. Our Theorem~\ref{thm:mainr1} below extends this to the case where $p$ is a prime of split multiplicative reduction.
 \end{rem}
\subsubsection{Definition of the normalization factor $\lambda_{\textup{BK}}$}
\label{subsec:definelambdaBK}
In this section we determine an element $\lambda_{\textup{BK}}$ which will have the property outlined in Remark~\ref{rem:whatislambdaBK}. 

Until the end of this article we assume that Nekov\'a\v{r}'s $p$-adic height pairing is non-degenerate. According to \cite[11.4.9]{nek} this is equivalent to asking that Schneider's height pairing is non-degenerate.
\begin{define} Let $\alpha_{\textup{BK}}:=\alpha(\frak{z}_0^{\textup{BK}}) \in \ZZ_p$ be given as in Definition~\ref{def:alphax}. It follows from Proposition~\ref{prop:locpkatononzero} that $\alpha_{\textup{BK}}$ is non-zero.
\end{define}
\begin{define}
Set $\lambda_{\textup{BK}}=\textup{ord}_p(\frak{C}_0)\cdot\left(\displaystyle{\frac{1}{\alpha_{\textup{BK}}}-\frac{\alpha_{\textup{BK}}}{\widetilde{h}_p(\frak{z}_0^{\textup{BK}})}\cdot \frac{\al}{\textup{ord}_p(q_E)}}\right)\,.$ Here $\al={\displaystyle\frac{\log_p(q_E)}{\textup{ord}_p(q_E)}}$ is the Mazur-Tate-Teitelbaum $\al$-invariant and 
$$\widetilde{h}_p(\frak{z}_0^{\textup{BK}}):\log_p\circ \rho_{\textup{cyc}}\left(\langle \frak{s}(\frak{z}_0^{\textup{BK}})\,,\, \frak{s}(\frak{z}_0^{\textup{BK}})\rangle_{\textup{Nek}}\right)$$ 
is the Nekov\'a\v{r}-height of $\frak{s}(\frak{z}_0^{\textup{BK}})$. Note that our normalization here differs from that of Section~\ref{subsec:compuationr0} by the factor $-E_p(1)$. 
\end{define}
The definition of $\lambda_{\textup{BK}}$ makes sense since we assumed that the $p$-adic height pairing (in particular its restriction to the image of $\frak{s}$ which compares with the classical $p$-adic height pairings) is non-degenerate. 
\begin{define}
\label{def:nekovarliftofkato}
We define the normalization of the Beilinson-Kato element as the lift
$$[(\frak{Z},(\frak{Z}_\ell^+),(\nu_\ell))]=\widetilde{\frak{Z}}_{\textup{BK}}=\frak{s}(\lambda_{\textup{BK}}\cdot\frak{z}_0^{\textup{BK}}) \in \widetilde{H}^1_f(V).$$ 
\end{define}
We shall prove below (Proposition~\ref{lambdaisthecorrectlambda}) that it is the sought after factor verifying 
the desired identity (\ref{eqn:reducetodercolemanahainstderbk}) and in Proposition~\ref{prop:lambdaisnonvanishing} that $\lambda_{\textup{BK}}\neq 0$. Let $\mathcal{D}(\frak{z}_\infty^{\textup{BK}})$ be the class denoted by $[Dx_{\textup{Iw}}]$ in Lemma~\ref{lem:bocksteinnormalizedderivative} (attached to the data $[x_\textup{Iw}]=\frak{z}_\infty^{\textup{BK}}$ and $[x_f]=\frak{s}(\frak{z}_0^{\textup{BK}})$).
 \begin{prop}
 \label{lambdaisthecorrectlambda}
 The element $\frak{C}_0 -\frak{Z}_p^+ \in H^1(\QQ_p,F_p^+V)$ annihilates $\mathcal{D}(\frak{z}_\infty^{\textup{BK}}) \in H^1(\QQ_p,F_p^-V)$ under the Tate pairing.
 \end{prop}
 \begin{proof}
 To ease notation we will write ${h}=\widetilde{h}_p(\frak{z}_0^{\textup{BK}})$, $\lambda_{\textup{BK}}=\textup{ord}_p(\frak{C}_0)\cdot\lambda$ and set
 $$\xymatrix@R=.3pc{D: H^1(\QQ_p,F_p^+V)=\widehat{\QQ}_p^\times\otimes\QQ_p \ar[r]& \QQ_p\\
 c\ar@{|->}[r]&\left\langle \mathcal{D}\left(\frak{z}_0^{\textup{BK}}\right),c \right\rangle_{\textup{Tate}}}
 $$
 and we will prove that 
 \be\label{eqn:whatwewanttoprove}
 D\left(\textup{ord}_p(q_E) \left(\frak{C}_0 -\frak{Z}_p^+\right) \right)=0\,.
 \ee
 Note that it follows from Corollary~\ref{cor:incaseallbutoneacyclic}, Theorem~\ref{thm:heightformular0}(ii) and the choice of the lift $\frak{s}(\frak{z}_0^{\textup{BK}})$ in Section~\ref{subsubsec:liftBKtotilde} that
 \be\label{eqn:Donq} D(q_E\otimes1)=\al\cdot \alpha_{\textup{BK}}
 \ee
and similarly that
\be
\label{eqn:Donu}
D(u_E\otimes1)=\frac{\textup{ord}_p(q_E)\,{h}}{\alpha_{\textup{BK}}}\,.
\ee
Using (\ref{eqn:Donq}) and (\ref{eqn:Donu}) together with the definition of $\lambda$, we infer that
\begin{align*}D(q_E\otimes1+u_E\otimes\lambda\alpha_{\textup{BK}})&=\al\cdot \alpha_{\textup{BK}}+\lambda\cdot\textup{ord}_p(q_E)\cdot h\\
&=\al\cdot \alpha_{\textup{BK}}+\left(\frac{1}{\alpha_{\textup{BK}}}-\frac{\alpha_{\textup{BK}}}{h}\cdot \frac{\al}{\textup{ord}_p(q_E)}\right)\cdot\textup{ord}_p(q_E)\cdot h\\
&=D(u_E\otimes 1)\,.
\end{align*}
This shows that
\be\label{eqn:andwegotwhatwewant}
D\left(q_E\otimes1+u_E\otimes(\lambda\alpha_{\textup{BK}}-1)\right)=0\,.
\ee
The desired equality (\ref{eqn:whatwewanttoprove}) follows on noticing that
$$\textup{ord}_p(\frak{C}_0)\cdot\left(q_E\otimes1+u_E\otimes(\lambda\alpha_{\textup{BK}}-1)\right)=\textup{ord}_p(q_E)\cdot\left(\frak{C}_0 -\frak{Z}_p^+\right)\,.$$
 \end{proof}
 \begin{prop}
 \label{prop:lambdaisnonvanishing}
 $\lambda_{\textup{BK}}\neq0$.
 \end{prop}
 \begin{proof}
For a set $S=\{z_1,\cdots,z_n\}\subset \widetilde{H}^1_f(V)$ let $R_S=\det\left(\langle z_i,z_j\rangle_{\textup{Nek}}\right)$ denote the $p$-adic regulator of this set $S$. 
\begin{align*}\lambda_{\textup{BK}}=0 &\iff \widetilde{h}_p(\frak{z}_0^{\textup{BK}})\cdot\textup{ord}_p(q_E)=\alpha_{\textup{BK}}^2\cdot\al\\
&\iff \widetilde{h}_p(\frak{z}_0^{\textup{BK}})/\alpha_{\textup{BK}}^2\cdot\log_p(q_E)=\al^2\\
\tag{$\star$}&\iff  \widetilde{h}_p(\frak{z}_0^{\textup{BK}}/\alpha_{\textup{BK}})\cdot \widetilde{h}_p\left([1]\right)=\left\langle [1], \frak{s}\left(\frak{z}_0^{\textup{BK}}/\alpha_{\textup{BK}} \right) \right\rangle_{\textup{Nek}}^2\\
&\iff R_{\frak{B}}=0
\end{align*}
where ($\star$) follows from Theorem~\ref{thm:heightformular0} (i)-(ii) and where $\frak{B}=\{[1],\frak{s}(\frak{z}_0^{\textup{BK}})\}$. By Theorem~\ref{thm:rodolfopr} the set $\frak{B}$ is $\QQ_p$-linearly independent so the equality $R_{\frak{B}}=0$ contradicts the non-degeneracy of Nekov\'a\v{r}'s $p$-adic height pairing. 

 \end{proof}
\subsubsection{The height of the normalized Beilinson-Kato elements}
\label{subsubsec:theheightofnormalizedBK}
We are ready to prove the main result of this article. Set 
$$\Xi_n:=\textup{res}_p^s(\frak{z}_{n}^{\textup{BK}}) \in H^1(\Phi_n,F_p^-T^*)$$ 
and let $\Xi=\{\Xi_n\} \in H^1(\QQ_p,F_p^-T^*\otimes\LL)$. Note that we are once again implicitly identifying $T$ with $T^*$. Our running assumptions show that $\Xi_0=0$ and this fact allows us to choose $\Xi^\prime=\{\Xi_n^\prime\} \in H^1(\QQ_p,F_p^-T^*\otimes\LL)$ as in Lemma~\ref{lem:derexists} (applied with $X=F_p^-T^*$).
\begin{define}
\label{def:derivedmeasure}
Let $\mu_E \in \Lambda$ be the element defined as
$$\mu_E=\left\{\sum_{\tau\in \Gamma_n}\langle \frak{C}_n^\tau,\Xi_{n}^\prime\rangle_{\textup{Tate}}\cdot\tau\right\} \in \varprojlim \ZZ_p[\Gamma_n]\,.$$
Although $\mu_E$ depends on the choice of $\Xi^\prime$ and $\gamma$, the value
\be\label{eqn:mainrankone}\int_{\Gamma}\pmb{1} \cdot d\mu_E=\pmb{1}(\mu_E)=\langle \frak{C}_0,\Xi_{0}^\prime\rangle_{\textup{Tate}}\ee
does not, as shown by Lemma~\ref{lem:calculateder}.
\end{define}
Recall that $J=\ker(\LL\ra\ZZ_p)$ is the augmentation ideal.
\begin{prop}
\label{prop:congmodj3}
$\displaystyle{\frac{(\gamma-1)^2}{\log_p(\rho_{\textup{cyc}}(\gamma))^2}}\,\,\mu_E\equiv \al_E \mod J^3.$
\end{prop}
\begin{proof}
Let $\al_E^\prime\in \LL$ be the element
$$\al_E^\prime:=\left\{\sum_{\tau\in \Gamma_n}\langle \frak{C}_n^\tau,\Xi_{n}\rangle_{\textup{Tate}}\cdot\tau\right\}\,.$$
Recall that $\displaystyle{\al_E=\left\{\sum_{\tau\in \Gamma_n}\langle d_n^\tau,\Xi_{n}\rangle_{\textup{Tate}}\cdot\tau\right\}}$ is the Mazur-Tate-Teitelbaum $p$-adic $L$-function, as explained in \cite[Section 4]{kobayashi06}. Lemma~\ref{lem:normalizedderivedmeasure} shows that
$$\displaystyle{\frac{(\gamma-1)}{\log_p(\rho_{\textup{cyc}}(\gamma))}}\,\,\al_E^\prime\equiv \al_E \mod J^2,$$
and also that
$$\displaystyle{\frac{(\gamma-1)}{\log_p(\rho_{\textup{cyc}}(\gamma))}}\,\,\mu_E\equiv \al_E^\prime \mod J^2.$$
\end{proof}
Recall $L_p(E,s)=\rho_{\textup{cyc}}^{s-1} (\al_E)$ and the generator $\gamma_0\in \Gamma$ that satisfies $\log_p(\rho_{\textup{cyc}}(\gamma_0))=p$.
\begin{prop}
\label{prop:derivedmeasurevsderivative}
$\displaystyle{\frac{d^2}{ds^2}\left(L_p(E,s)\right)\big|_{s=1}=2\cdot\langle \frak{C}_0,\Xi_{0}^\prime\rangle_{\textup{Tate}}}$.
\end{prop}
\begin{proof}
This follows from Proposition~\ref{prop:congmodj3} and (\ref{eqn:mainrankone}).
\end{proof}
\begin{rem}
\label{rem:comparetokobayashi}
The equality proved in Proposition~\ref{prop:derivedmeasurevsderivative} should be considered as the extension of the displayed equality (2) in \cite[p. 574]{kobayashi06}, to the case $r_{\textup{an}}(E)=1$.
\end{rem}
\begin{thm}
\label{thm:mainr1}
We have the following equality in $\Gamma$:
$$\frac{1}{2}\left(\frac{d^2}{ds^2}(L_p(E,s))\Big|_{s=1}\right)\otimes \gamma_0=-\frac{1}{\lambda_{\textup{BK}}}\cdot\langle\widetilde{\frak{Z}}_{\textup{BK}},\widetilde{\frak{Z}}_{\textup{BK}} \rangle_{\textup{Nek}}\,.$$
\end{thm}

\begin{proof}
Recall that $\Xi_n:=\textup{res}_p^s(\frak{z}_{n}^{\textup{BK}})$ and $\Xi:=\{\Xi_n\} \in H^1(\QQ_p,F_p^-T^*\otimes\LL)$. By the discussion at the start of Section~\ref{subsubsec:theheightofnormalizedBK}, we have an element
$$\Xi_0^\prime\in H^1(\QQ_p, F_p^-T^*)\otimes \Gamma$$
(defined up to an element of $H^0(\QQ_p,F_p^-T)\otimes\Gamma$) with the following properties:
\begin{itemize}
\item[(A)] $\left\langle\Xi_0^\prime,\frak{C}_0\right\rangle_{\textup{Tate}}=\left\langle \mathcal{D}(\frak{z}_\infty^{\textup{BK}}),\frak{C}_0\right\rangle_{\textup{Tate}}$\,, where $\mathcal{D}(\frak{z}_\infty^{\textup{BK}})$ is the element described in the paragraph following Definition~\ref{def:nekovarliftofkato}. (This follows from Lemma~\ref{lem:calculateder} using the fact that $\frak{C}_0$ is a universal norm, see also the comment following (\ref{eqn:derivativescompared})).
\item[(B)] $\displaystyle{\frac{1}{\lambda_{\textup{BK}}}}\left\langle\widetilde{\frak{Z}}_{\textup{BK}},\widetilde{\frak{Z}}_{\textup{BK}} \right\rangle_{\textup{Nek}}=-\left\langle\Xi_0^\prime,\frak{C}_0\right\rangle_{\textup{Tate}}$\,. (This follows from Corollary~\ref{cor:incaseallbutoneacyclic} combined with Proposition~\ref{lambdaisthecorrectlambda} and (A).)
\end{itemize}
The proof of the theorem follows from Proposition~\ref{prop:derivedmeasurevsderivative}\,.
\end{proof}
\begin{cor}
\label{cor:only}
Assuming Nekov\'a\v{r}'s height pairing is non-degenerate,
$$\textup{ord}_{s=1}\,L_p(E,s)=1+r_{\textup{an}}(E)$$
 when $r_{\textup{an}}(E)=0,1$.
\end{cor}
\begin{proof}
The assertion is due to Greenberg-Stevens~\cite{greenbergstevens} (without the assumption on Nekov\'a\v{r}'s heights) when $r_{\textup{an}}(E)=0$. The case $r_{\textup{an}}(E)=1$ follows from Theorem~\ref{thm:mainr1} and \cite[Proposition 11.4.9]{nek}, which reduces the non-degeneracy of the height pairing $\langle\,,\,\rangle_{\textup{Nek}}$ to the non-degeneracy of its restriction to $\frak{s}\left(\textup{Sel}_p(\QQ,V)\right)\otimes \frak{s}\left(\textup{Sel}_p(\QQ,V^*)\right)$, where both $\textup{Sel}_p(\QQ,V)$ and $\textup{Sel}_p(\QQ,V^*)$ are $\QQ_p$-vector spaces of dimension one.
\end{proof}

\appendix
\section{A Rubin-style formula}
First version of this article has been circulated late 2012. That version and all others up until now relied on Nekov\'a\v{r}'s higher Rubin-style formula proved in \cite[Proposition 11.5.11]{nek}. As indicated by Venerucci, this proposition is flawed (to our embarrassment, we have missed out on that comment to this day). 
The goal in this appendix is to prove a corrected version of Nekov\'a\v{r}'s claim befitting our needs in this current article. The notation we use here is borrowed from Section~\ref{subsubsec:selmer}. We work and prove our statements in for a general Galois representation $X$ as in Section~\ref{subsubsec:selmer} except that we work over $K=\QQ$.


Consider the exact sequence 
$$0\lra X\otimes \Gamma \lra X\otimes \LL/J^2 \lra X\lra 0$$
where the map $X\otimes \Gamma \ra X\otimes \LL/J^2$ is given by the map $x\otimes \gamma \mapsto x\otimes (\gamma-1)$\,. This induces a commutative diagram of complexes
$$\xymatrix{
& 0\ar[d]&0\ar[d]& 0\ar[d]&\\
0\ar[r]& \widetilde{C}^\bullet_f(X\otimes\Gamma) \ar[d]\ar[r]&{C}^\bullet(X\otimes\Gamma) \ar[r]\ar[d]& U_S^-(X\otimes\Gamma) \ar[d]\ar[r]& 0\\
0\ar[r]& \widetilde{C}^\bullet_f(X\otimes \LL/J^2) \ar[d]\ar[r]&{C}^\bullet(X\otimes \LL/J^2) \ar[d]\ar[r]& U_S^-(X\otimes \LL/J^2)\ar[d] \ar[r]& 0\\
0\ar[r]& \widetilde{C}^\bullet_f(X)\ar[d] \ar[r]&{C}^\bullet(X) \ar[d]\ar[r]& U_S^-(X) \ar[d]\ar[r]& 0\\
& 0&0& 0& 
}$$
where $\widetilde{C}^\bullet_f(Z)$ is the short for $\widetilde{C}^\bullet_f(G_{K,S},Z,\Delta(Z))$ for $Z=X, X\otimes \LL/J^2$ and  the horizontal (exact) lines are deduced from \cite[(6.1.3.1)]{nek}. In the level cohomology this gives rise to the commutative diagram (as in \cite[Lemma 1.2.19]{nek}) 
$$\xymatrix@R=2pc @C=1.3pc {&&&&H^0(U_S^-(X))\ar[d]^{\beta^0}\\
&&&&H^1(U_S^-(X))\otimes\Gamma\ar[d]^{i}\\
&&&H^1(G_{\QQ,S},X\otimes\LL/J^2)\ar[d]^{\textup{pr}}\ar[r]^(.5){\textup{res}_S^-}&H^1(U_S^-(X\otimes \LL/J^2))\ar[d]^{\textup{pr}}\\
&H^0(U_S^-(X))\ar[d]_{\beta^0}\ar[r]&\widetilde{H}^1_f(X)\ar[d]^{\beta^1}\ar[r]&H^1(G_{\QQ,S},X)\ar[r]_(.5){\textup{res}_S^-}\ar[d]&H^1(U_S^-(X))\ar[d]\\
&H^1(U_S^-(X))\otimes\Gamma\ar[r]_(.55){\partial}&\widetilde{H}^2_f(X)\otimes\Gamma\ar[r]&H^2(G_{\QQ,S},X)\otimes\Gamma\ar[r]&H^2(U_S^-(X))\otimes\Gamma
}$$
where $\textup{pr}$ is the map induced from the augmentation map $\LL\ra\ZZ_p$.

\begin{lemma}
\label{lem:bocksteinnormalizedderivative}
Suppose that we are given a class $[x_f]=[(x,(x_\ell^+)_{\ell \in S},(\lambda_\ell)_{\ell \in S})] \in \widetilde{H}^1_f(X)$ such that 
$$\textup{pr}\left([x_{\textup{Iw}}]\right)=[x] \in H^1(G_{\QQ,S},X)$$
for some $[x_{\textup{Iw}}]\in H^1(\QQ,X\otimes\LL)$.
Then there exists a class 
$$[Dx_{\textup{Iw}}]=\left([Dx_{\textup{Iw}}]_{\ell}\right)_{\ell \in S} \in \bigoplus_{\ell\in S} H^1(U_\ell^-(X))\otimes\Gamma= H^1(U_S^-(X))\otimes\Gamma$$
 such that 
\begin{itemize}
\item[(i)] $i([Dx_{\textup{Iw}}])=\textup{res}_S^-\left([x_{\textup{Iw}}] \,\textup{mod} \,J^2\right)\,.$\\
\item[(ii)] $\beta^1([x_f])=-\partial([Dx_{\textup{Iw}}])$\,.
\end{itemize}
\end{lemma}
\begin{proof}
Since we have 
$$\textup{pr}\circ\textup{res}_S^-\left([x_{\textup{Iw}}]\,\textup{mod}\, J^2\right)=\textup{res}_S^-([x])=0$$ 
it follows from Lemma~1.2.19 of \cite{nek} that there exists a class $D \in H^1(U_S^-(X))\otimes\Gamma$ that verifies 
\begin{itemize}
\item $i(D)=\textup{res}_S^-\left([x_{\textup{Iw}}] \,\textup{mod}\, J^2\right)$\,,\\
\item $\beta^1([x_f])+\partial(D)-\partial\circ\beta^0(t)=0$ for some $t\in H^0(U_S^-(X))$\,.
\end{itemize}
Set $[Dx_{\textup{Iw}}]=D-\beta^0(t)$.
\end{proof}
We will call the class $[Dx_{\textup{Iw}}]$ the \emph{Bockstein-normalization of the derivative of} of $x_{\textup{Iw}}$. The main goal in this appendix is to give a proof of the following Proposition, which we refer to as the \emph{Rubin-style-formula}.
\begin{prop}
\label{propappendix:RSformula}
Let $[y_f]=[(y,(y_\ell^+)_{\ell \in S},(\mu_\ell)_{\ell \in S})] \in \widetilde{H}^1_f(X^*)$ be any class. For $[x_f]$ and $[x_{\textup{Iw}}]$ as in the statement of Lemma~\ref{lem:bocksteinnormalizedderivative} we have
$$\langle [x_f], [y_f]  \rangle_{\textup{Nek}}=-\sum_{\ell\in S}\textup{inv}_\ell\left((Dx_{\textup{Iw}})_\ell\cup y_\ell^+\right)\,.$$
\end{prop}
\begin{rem}
\label{rem:Appnoambiguity}
The reader might feel uneasy that the expression $(Dx_{\textup{Iw}})_\ell$ is defined only up to a coboundary $d(u,v) \in dU_S^-(X)^0\otimes\Gamma$. We check here that 
$$\sum_{\ell\in S}\textup{inv}_\ell\left(d(u,v)\cup y_{\ell}^+\right)=0$$
and therefore verify that the right side of the asserted equality in Proposition~\ref{propappendix:RSformula} is well-defined. Indeed,
\begin{align*}
\textup{inv}_\ell\left(d(u,v)\cup y_{\ell}\right)&=\textup{inv}_\ell\left((du,-i_{\ell}^+(u)+dv)\cup y_{\ell}^+\right)\\
&=\textup{inv}_\ell\left((-i_{\ell}^+(u)+dv)\cup i_\ell^+(y_\ell^+) +\frak{h}_\ell(du\otimes y_\ell^+)\right)\\
&=\textup{inv}_\ell\left((-i_{\ell}^+(u)+dv)\cup i_\ell^+(y_\ell^+) \right)\\
&=\textup{inv}_\ell\left(dv\cup i_\ell^+(y_\ell^+) \right)\\
&=\textup{inv}_\ell\left(dv\cup(\textup{res}_{\ell}(y)+d\mu_\ell) \right)\\
&=\textup{inv}_\ell\left(dv\cup\textup{res}_{\ell}(y) \right)\\
&=\textup{inv}_\ell\left(d(v\cup\textup{res}_{\ell}(y)) \right)
\end{align*}
vanishes. Here, 
\begin{itemize}
\item $\frak{h}_\ell: U_S^+(X)^2\otimes U_S^+(X^*(1))^1\ra \bigoplus_{\ell\in S} \tau_{\geq 2} C^\bullet(G_\ell,R(1))$ is a null-homotopy to the cup-product pairing
$$U_S^+(X)\otimes U_S^+(X^*(1)) \stackrel{i_\ell^+\otimes i_\ell^+}{\lra} \bigoplus_{\ell\in S} C^\bullet(G_\ell,X)\otimes  \bigoplus_{\ell\in S} C^\bullet(G_\ell,X^*(1))\stackrel{\tau_{\geq2}\circ\cup}{\lra}\oplus_{\ell\in S}\,\tau_{\geq 2}  C^\bullet(G_\ell,R(1))\,,$$
\item the second equality is due to Nekov\'a\v{r}'s computations in \cite[\S6.2.2]{nek},
\item the last equality is because 
$$d(v\cup\textup{res}_{\ell}(y))=dv\cup \textup{res}_{\ell}(y) -v\cup \textup{res}_{\ell}(dy)=dv\cup \textup{res}_{\ell}(y)$$
since $dy=0$.
\end{itemize}
\end{rem}
Before we proceed with the proof of Proposition~\ref{propappendix:RSformula}, we first prove some auxiliary statements (which are mostly trivial in the special case when $H^0(U_S^-(X))=0$).

Set $\iota_p^-: C^\bullet (G_p, X) \ra C^\bullet (G_p, F^-_p X)$ so that we have an exact sequence of complexes 
$$0\lra C^\bullet(G_p, F^+_p X)\stackrel{i_p^+}{\lra} C^\bullet (G_p, X) \stackrel{\iota_p^-}{\lra} C^\bullet (G_p, F^-_p X)\lra0\,.$$
\begin{define}
\label{define:therrorcomplex}
Define the complex $C^\bullet_+(X)$ by setting
$$C^i_+(X)=C^{i+1}(G_p,F^+_pX)\oplus C^i(G_p,F_p^+X)$$
(and $C^0_+(X)=C^0(G_p,F^+_pX)\oplus 0$) with differentials $d_{i}(C^\bullet_+)=\left(\begin{array}{cc} d_{i+1} & 0 \\ -\textup{id} & d_i\end{array}\right)$\,. 
\end{define}
\begin{lemma}
\label{lemma:errorcomplexacyclic}
The complex $ C^\bullet_+(X)$ is acyclic.
\end{lemma}
\begin{proof}
Direct computation of cocycles and coboundaries.
\end{proof}
\begin{lemma}
\label{lemma:minuscohomologycomparison}
The morphism $j_p^-:\, U_p^-(X)\ra C^\bullet (G_p,F^-_p X)$
given explicitly in degree $i \geq 1$ by 
$$\xymatrix{j_p^-:\, U_p^-(X)^i=C^{i+1}(G_p,F^+_pX)\oplus C^i(G_p,X)\ar[r]^(.72){(0, \iota_p^-)}& C^i (G_p,F^-_p X)
}$$(resp., in degree $\leq 0$ by the $zero$-map) is a quasi-isomorphism. Furthermore the morphism 
$\iota_p^-$ factors through the quasi-isomorphism $j_p^-$.
\end{lemma}
\begin{proof}
The following sequence of complexes 
$$0\lra C^\bullet_+(X)\stackrel{(\textup{id},i_p^+)}{\lra} U_p^-(X)\lra C^\bullet (G_p,F_p^-X)\lra 0$$
is exact (where the complex $C^\bullet_+(X)$ is as given in Definition~\ref{define:therrorcomplex}). For every $i \geq 0$, we therefore have in the level of cohomology an exact sequence 
$$H^i(C^\bullet_+(X)) \lra H^i(U_S^-(X))\lra H^1(G_p,F_p^-X) \lra H^{i+1}(C^\bullet_+(X))$$
The first assertion now follows from Lemma~\ref{lemma:errorcomplexacyclic}.

The second assertion is obvious from definitions.
\end{proof}
\begin{define}
\label{define:minusforallcochains}
For $\alpha=(\alpha_\ell) \in \bigoplus_{\ell \in S} C^1(G_\ell,X)$ let 
$$\alpha^-=(0,\alpha)\in U_S^-(X)^1=U_S^+(X)^2\oplus  \bigoplus_{\ell \in S} C^1(G_\ell,X)$$ 
denote its \emph{singular projection}.
\end{define}
\begin{define}
\label{define:minuscohomology}
Suppose $\alpha=(\alpha_\ell) \in \bigoplus_{\ell \in S} C^1(G_\ell,X)$ is a cochain such that $d\alpha=i_S^+(b)$ for a (unique) 
cocycle $b \in U_S^+(X)^2$, $db=0$. Then the cochain
$$\alpha^-_b:=(b,\alpha) \in U_S^-(X)^1$$  
is a cocycle. We denote its class in $H^1(U_S^-(X))$ by $[\alpha^-_b]$ and call it the \emph{singular projection} of $\alpha$. Furthermore, the cochain $\partial\left(b,\alpha\right) \in \widetilde{C}^2_f(X)$which is the image of the cocycle $(b,\alpha)$ under the natural morphism
$$U_S^-(X)^1\stackrel{\partial}{\lra} \widetilde{C}^2_f(X) =C^2(G_{{\QQ},S},X) \oplus U_S^-(X)^1$$
is the cocycle $(0,b,\alpha) \in \widetilde{Z}^2_f(X)$. Attached to $\alpha$ as above, we therefore have a uniquely determined class $\partial([\alpha^-_b]) \in \widetilde{H}^2_f(X)$.
\end{define}
Note that if $\alpha$ above is itself a cocycle, then we can take $b=0$ and $[\alpha^-]$ is simply the image of this cocycle in $H^1(U_S^-(X))$ (under the obvious map $\bigoplus_{\ell \in S} C^1(G_\ell,X) \ra U_S^-(X)^1$) and the construction in Definition~\ref{define:minuscohomology} slightly extends this notion. 
\begin{define}(Global Cup Product Pairing)
\label{define:globalcupproducpairing} Let 
$$\langle \,,\,\rangle_{PT}: \widetilde{H}^2_f(X)\otimes\widetilde{H}^1_f(X^*(1))\lra R$$
denote Nekov\'a\v{r}'s cup-product pairing whose definition we briefly recall here. Nekov\'a\v{r} first defines a cup product (up to homotopy)
$$\cup: \widetilde{C}^2_f(X)\otimes\widetilde{C}^1_f(X^*(1))\lra C^3_c(G_{\QQ,S},\ZZ_p(1))\otimes R$$
where 
$$C^\bullet_c(G_{\QQ,S},\ZZ_p(1))=\textup{cone}\left(C^\bullet(G_{\QQ,S},\ZZ_p(1))\stackrel{\textup{res}_S}{\lra}C^\bullet(G_p,\ZZ_p(1))\right)$$ denotes the compactly supported cochains. In the level of cohomology, this yields a pairing 
$$\cup: \widetilde{H}^2_f(X)\otimes\widetilde{H}^1_f(X^*(1))\lra H^3_c(G_{\QQ,S},\ZZ_p(1))\otimes R\,.$$
This pairing is given explicitly as follows. Let $[z_f]=[\left(z,z_S^+=(z_\ell^+),\omega_S=(\omega_\ell)\right)] \in  \widetilde{H}^2_f(X)$ and $y_f=\left[\left(y, y_S^+=(y_\ell^+),\mu_S=(\mu_\ell)\right)\right]\in \widetilde{H}^1_f(X^*(1))$. Then 
$$z_f\cup y_f =(z\cup y, \omega_S\cup \textup{res}_S(y) + i_S^+(z_S^+)\cup \mu_S) \in C^3_c(G_{\QQ,S},R(1))\,.$$
Using the cup product pairing, the Poitou-Tate Global pairing is then given using the reciprocity law of class field theory. The definition of $C^\bullet_c(G_{\QQ,S},\ZZ_p(1))$ and global class field theory gives rise to the following diagram with exact rows:
$$\xymatrix{H^2(G_{\QQ,S},\ZZ_p(1))\ar[r]^(.48){\textup{res}_S}\ar@2{-}[d]&\oplus_{\ell \in S}H^2(G_\ell,\ZZ_p(1)) \ar[r]^(.52){\partial_c}\ar@2{-}[d]& H^3_c(G_{\QQ,S},\ZZ_p(1))\ar[r]\ar@{.>}[d]^{\textup{inv}_S}& 0\\
H^2(G_{\QQ,S},\ZZ_p(1))\ar[r]^(.48){\textup{res}_S}&\oplus_{\ell \in S}H^2(G_\ell,\ZZ_p(1)) \ar[r]^(.65){\sum_{\ell\in S}\textup{inv}_\ell}& \ZZ_p \ar[r]& 0}$$
Furthermore one may compute $\textup{inv}_S([z_f\cup y_f])$ to be 
$$\textup{inv}_S([z_f\cup y_f])=\sum_{\ell\in S}\textup{inv}_{\ell}\left(\omega_\ell\cup \textup{res}_\ell(y) + i_\ell^+(z_\ell^+)\cup \mu_\ell + \textup{res}_{\ell}(W)\right)$$
where $W \in C^2(G_{\QQ,S},R(1))$ is \emph{any} cochain such that $dW=z\cup y$. (The existence of such an element is guaranteed since $z\cup y$ is a cocycle and the cohomological dimension of $G_{\QQ,S}$ is 2.) See \cite[(11.3.11.1)]{nek} for a proof of this statement. Note that none of the assumptions (beyond what is in effect here) of Lemma 11.3.11 in loc.cit. is required to deduce (11.3.11.1). We now set 
$$\langle [z_f],[y_f]\rangle_{PT}:=\textup{inv}_S\left([z_f\cup y_f)]\right)=\sum_{\ell\in S}\textup{inv}_{\ell}\left(\omega_\ell\cup  \textup{res}_\ell(y) + i_\ell^+(z_\ell^+)\cup \mu_\ell + \textup{res}_{\ell}(W)\right)\,.$$
\end{define}
\begin{lemma}
\label{lemma:useglobalCFT}
In the notation of Definition~\ref{define:globalcupproducpairing}, suppose that $z=dZ$ is a coboundary. Then 
$$\langle [z_f],[y_f]\rangle_{PT}=\sum_{\ell\in S}\textup{inv}_{\ell}\left(\omega_\ell\cup \textup{res}_\ell(y) + i_\ell^+(z_\ell^+)\cup \mu_\ell\right)$$\,.
\end{lemma}
\begin{proof}
Note that $d(Z\cup y)=dZ\cup y$ since $y$ is a cocycle. We may therefore choose $W=Z\cup y$ (which is a cocycle)
$$\sum_{\ell\in S}\textup{inv}_{\ell}\left(\textup{res}_{\ell}(z\cup y)\right)=\sum_{\ell\in S}\textup{inv}_{\ell}\left(\textup{res}_{\ell}(Z\cup y)\right)=0$$
by global class field theory.
\end{proof}
\begin{proof}[Proof of Proposition~\ref{propappendix:RSformula}]
We essentially do no more than expanding on Nekov\'a\v{r}'s proof of \cite[Proposition 11.3.15]{nek}, clarifying certain points and adapting them to our needs. We still opt to include all details to make sure that the argument goes through without any trouble. 

Let $\beta^1(x_f)=\left(z,z_S^+=(z_\ell^+)_{\ell\in S},\omega_S=(\omega_\ell)_{\ell \in S}\right)\in\widetilde{C}^2_f(X)\otimes\Gamma$. Consider the diagram
$$\xymatrix@R=2pc @C=1.3pc {&H^1(G_{\QQ,S},X\otimes\LL/J^2)\ar[d]^{\textup{pr}}\\
\widetilde{H}^1_f(X)\ar[d]^{\beta^1}\ar[r]^(.45){\pi_1}&H^1(G_{\QQ,S},X)\ar[d]^{\beta^1}\\
\widetilde{H}^2_f(X)\otimes\Gamma\ar[r]^(.45){\pi_2}&H^2(G_{\QQ,S},X)\otimes\Gamma
}$$
Since $\pi_1([x_f])=[x]=\textup{pr}([x_{\textup{Iw}}]\, \textup{mod}\, J^2) \in \ker(\beta_1)$, it follows from the commutativity of the square in the diagram that 
$$\pi_2(\beta^1([x_f]))=[z]=0$$
as well. Hence $z=dZ$ (where $Z\in C^1(G_{\QQ,S},X)\otimes\Gamma$) is a coboundary. Furthermore, since $\beta^1(x_f)$ is a cocycle it follows that 
$$i_S^+(z_S^+)=d\omega_S+\textup{res}_S(z)=d(\omega_S+\textup{res}_S(Z)).$$
Applying the discussion of Definition~\ref{define:minuscohomology} (on replacing $X$ with $X\otimes\Gamma$ and setting $\alpha:=\omega_S+\textup{res}_S(Z)$, $b=z_S^+$) we conclude that 
$$\partial([\alpha^-_{z_S^+}])=\left[\left(0,z_S^+,\alpha=\omega_S+\textup{res}_S(Z)\right)\right] \in \widetilde{H}^1_f(X)\otimes\Gamma\,.$$
On the other hand 
$$\beta^1(x_f)=\left(z,z_S^+,\omega_S\right)=\left(0,z_S^+,\alpha\right)+(dZ,0,-\textup{res}_S(Z))=\left(0,z_S^+,\alpha\right)+d(Z,0,0)$$
i.e., $\beta^1(x_f)-\left(0,z_S^+,\alpha\right) \in \widetilde{Z}^1_f(X)$ is in fact a coboundary and hence 
$$\label{eqn:Rubinfirstreduction}\beta^1([x_f])=[(0,z_S^+,\alpha)]=\partial([\alpha^-_{z_S^+}])\,.$$
By Lemma~\ref{lemma:useglobalCFT}, 
$$\langle \beta^1([x_f]),[y_f] \rangle_{\textup{PT}}=\sum_{\ell \in S} \textup{inv}_\ell(\alpha_\ell \cup \textup{res}_\ell(y) +i_\ell^+({z_\ell^+})\cup\mu_\ell)$$ 
and thus
\begin{align*}
\left\langle \beta^1([x_f]),[y_f] \right\rangle_{\textup{PT}}&=\sum_{\ell \in S} \textup{inv}_\ell(\alpha_\ell \cup \textup{res}_\ell(y) +d\alpha_\ell\cup\mu_\ell) =\sum_{\ell \in S} \textup{inv}_\ell(\alpha_\ell \cup \textup{res}_\ell(y) +\alpha_\ell\cup d\mu_\ell)\\
&=\sum_{\ell \in S} \textup{inv}_\ell(\alpha_\ell \cup (\textup{res}_\ell(y) + d\mu_\ell))\\
&=\sum_{\ell \in S} \textup{inv}_\ell(\alpha_\ell \cup i_\ell^+(y_\ell^+))\\
&=\sum_{\ell \in S} \textup{inv}_\ell\left((\alpha_{z_S^+}^-)_{\ell} \cup y_\ell^+\right) \tag{\ding{168}}
\end{align*}
Here, the equality (\ding{168}) follows from the formulas in \cite[\S6.2.2]{nek}: 
$$(\alpha_{z_S^+}^-)_{\ell} \cup y_\ell^+=(z_\ell^+,\alpha_\ell)\cup y_\ell^+ :=\alpha_\ell\cup i_\ell^+(y_\ell^+) +\frak{h}_\ell(z_\ell\otimes y_\ell^+)$$
where $\frak{h}_\ell$ is a null-homotopy from Remark~\ref{rem:Appnoambiguity} and it vanishes under $\textup{inv}_\ell$.

Let $[Dx_{\textup{Iw}}] \in H^1(U_S^-(X))\otimes\Gamma$ be as in Lemma~\ref{lem:bocksteinnormalizedderivative}. Then 
$$-\partial([Dx_{\textup{Iw}}])=\beta^1([x_f])=\partial([\alpha^-_{z_S^+}])$$
The exactness of the sequence
$$H^1(G_{\QQ,S},X)\otimes\Gamma\stackrel{\textup{res}_S^-}{\lra} H^1(U_S^-(X))\otimes\Gamma\stackrel{\partial}{\lra} \widetilde{H}^2_f(X)\otimes\Gamma$$
shows that there is a cocycle $\frak{E}\in C^1(G_{\QQ,S},X)$, $d\frak{E}=0$ such that 
\begin{align*}U_S^-(X)^1\ni(0,\textup{res}_S(\frak{E}))=\textup{res}_S(\frak{E})^-&=Dx_{\textup{Iw}}+\alpha^-_{z_S^+}+d(u,v)
\end{align*}
where $(u,v)\in U_S^-(X)^1$. Combining this fact with the formula we obtained above for $\left\langle \beta^1([x_f]),[y_f] \right\rangle_{\textup{PT}}$ and making use of the observation in Remark~\ref{rem:Appnoambiguity} we have,
$$\left\langle \beta^1([x_f]),[y_f] \right\rangle_{\textup{PT}}=-\sum_{\ell \in S} \textup{inv}_\ell\left((Dx_{\textup{Iw}})_\ell \cup y_\ell^+\right)+\sum_{\ell \in S} \textup{inv}_\ell\left(\textup{res}_S(\frak{E})^- \cup y_\ell^+\right)\,.$$
However,
\begin{align}
\notag\sum_{\ell \in S} \textup{inv}_\ell\left(\textup{res}_\ell(\frak{E})^- \cup y_\ell^+\right)&=\sum_{\ell \in S} \textup{inv}_\ell\left(\textup{res}_\ell(\frak{E}) \cup i_{\ell}^+(y_\ell^+)\right)\\
\notag&=\sum_{\ell \in S} \textup{inv}_\ell\left(\textup{res}_\ell(\frak{E}) \cup (\textup{res}_{\ell}(y)+d\mu_{\ell})\right)\\
\notag&=\sum_{\ell \in S} \textup{inv}_\ell\left(\textup{res}_\ell(\frak{E}\cup y)\right)+\sum_{\ell \in S} \textup{inv}_\ell\left(\textup{res}_\ell(\frak{E}) \cup d\mu_{\ell}\right)\\
\tag{$\dagger$}&=\sum_{\ell \in S} \textup{inv}_\ell\left(\textup{res}_\ell(\frak{E}) \cup d\mu_{\ell}\right)\\
\tag{$\star$}&=-\sum_{\ell \in S} \textup{inv}_\ell\left(d\left(\textup{res}_\ell(\frak{E}) \cup \mu_{\ell}\right)\right) =0
\end{align}
where the first equality follows from the discussion of \cite[\S6.2.2]{nek}, ($\dagger$) by the reciprocity law and ($\star$) by using the fact that $d\frak{E}=0$ and observing
$$d\left(\textup{res}_\ell(\frak{E}) \cup \mu_{\ell}\right)=d\,\textup{res}_\ell(\frak{E}) \cup \mu_{\ell}-\textup{res}_\ell(\frak{E}) \cup d\mu_{\ell}\,.$$ 
The proof of Proposition~\ref{propappendix:RSformula} follows.
\end{proof}
\begin{cor}
\label{cor:incaseallbutoneacyclic} In the setting of Proposition~\ref{propappendix:RSformula} assume also that the complexes $C^\bullet(G_\ell,X)$ are acyclic for $\ell\neq p$. Then
$$\langle[x_f],[y_f] \rangle_{\textup{Nek}}=-\left\langle \left[j_p^-(Dx_{\textup{Iw}})\right]_p, [y_p^+] \right\rangle_{\textup{Tate}}\,.$$
\end{cor}

{\scriptsize
\bibliographystyle{halpha}
\bibliography{references}
}
\end{document}